\documentclass[11pt]{article}
\usepackage{amsmath, graphicx, amsfonts,amssymb, calrsfs}
\usepackage{amsmath}
\usepackage{amsfonts}
\usepackage{mathrsfs}
\usepackage{amsthm}
\usepackage{setspace} 
 \usepackage{color}
\usepackage{enumitem}
\usepackage{hyperref}

\hypersetup{
  colorlinks   = true, 
  urlcolor     = blue, 
  linkcolor    = blue, 
  citecolor   = red 
  }

\usepackage{chngcntr}
\counterwithin{equation}{section}
\def\be{\begin{equation}}
\def\ee{\end{equation}}
\def\bea{\begin{eqnarray}}
\def\eea{\end{eqnarray}}
\def\bt{\begin{theorem}}
\def\et{\end{theorem}}
\def\bl{\begin{lemma}}

\def\el{\end{lemma}}
\def\br{\begin{remark}}
\def\er{\end{remark}}
\def\bc{\begin{corollary}}
\def\ec{\end{corollary}}
\def\bd{\begin{definition}}
\def\ed{\end{definition}}
\def\bp{\begin{proposition}}
\def\ep{\end{proposition}}
\newtheorem{theorem}{Theorem}[section]
\newtheorem{lemma}{Lemma}[section]
\newtheorem{remark}{Remark}[section]
\newtheorem{proposition}{Proposition}[section]
\newtheorem{corollary}{Corollary}[section]

\newtheorem{definition}{Definition}[section]
\begin{document}

\title{On the polar Orlicz-Minkowski problems and the $p$-capacitary Orlicz-Petty bodies\footnote{Keywords:  Minkowski problems, Orlicz-Brunn-Minkowski theory, Orlicz-Minkowski problems, $p$-capacity, Orlicz-Petty bodies, variational functionals.} }

\author{Xiaokang Luo, Deping Ye and Baocheng Zhu}
\date{}
\maketitle

\begin{abstract}
In this paper, we propose and study the polar Orlicz-Minkowski problems:  under what conditions on a nonzero finite measure $\mu$ and a continuous function $\varphi:(0,\infty)\rightarrow(0,\infty)$, there exists a convex body $K\in\mathcal{K}_0$ such that
$K$ is an optimizer of the following optimization problems:
\begin{equation*}
\inf/\sup \bigg\{\int_{S^{n-1}}\varphi\big( h_L \big) \,d \mu: L \in \mathcal{K}_{0} \ \text{and}\   |L^\circ|=\omega_{n}\bigg\}.
\end{equation*} The solvability of the polar Orlicz-Minkowski problems is discussed under different conditions. In particular, under certain conditions on $\varphi,$ the existence of a solution is proved for a nonzero finite measure $\mu$ on $S^{n-1}$ which is not concentrated on any hemisphere of $S^{n-1}.$ Another part of this paper deals with the $p$-capacitary Orlicz-Petty bodies.   In particular,  the existence of the $p$-capacitary Orlicz-Petty bodies is established and the continuity of the $p$-capacitary Orlicz-Petty bodies is proved.  
 \vskip 2mm
\noindent  2010 Mathematics Subject Classification: 52A20, 52A38, 52A39, 52A40, 53A15.
 
\end{abstract}

\section{Introduction}
Let $\mathcal{K}_0$ be the set of all convex bodies in $\mathbb{R}^n$ with the origin $o$ in their interiors, i.e., $K\in\mathcal{K}_0$ is a convex compact subset of $\mathbb{R}^n$ such that $o\in \text{int}K,$ the interior of $K.$ For  $0\neq q\in \mathbb{R},$ the $L_q$ mixed volume of $K,L\in\mathcal{K}_0$ (see e.g. \cite{Lutwak1993,Ye2015}) can be defined as
\begin{equation}\label{1.4}
V_q(K,L)=\displaystyle{\frac{1}{n}\int_{S^{n-1}}\Big(\frac{h_L(u)}{h_K(u)}\Big)^qh_K(u)\,dS(K,u)},
\end{equation} where $S(K,\cdot)$ is the surface area measure of $K$ (see e.g. \cite{Aleksandrov1937,FJ1938}) and $h_L$ is the support function of $L$ defined on $S^{n-1},$ the unit sphere of $\mathbb{R}^n$ (see Section \ref{section:2} for more details on the notations). When $q=1$, one gets the classical mixed volume (see e.g. \cite{Gruber2007,Schn2014}).  The classical and $L_q$ Minkowski inequalities \cite{Gardner2002,Gruber2007, Lutwak1993,Schn2014} lie at the heart of the rapidly developing $L_q$ Brunn-Minkowski theory of convex bodies. These inequalities read:  for all $q\geq1$ and all $K,L\in\mathcal{K}_0,$
\begin{equation}\label{1.5}
V_q(K,L)\geq |K|^{\frac{n-q}{n}}|L|^{\frac{q}{n}};
\end{equation}
 equality holds if and only if $K$ is dilate of $L$ for $q>1$ and  $K$ is homothetic to $L$ for $q=1$. Hereafter $|K|$ refers to the volume of $K\in\mathcal{K}_0$ and $\omega_n=|B^n_2|$ denotes the volume of the unit ball $B^n_2$ in $\mathbb{R}^n.$ The $L_q$ Minkowski inequality (\ref{1.5}) implies that, for any fixed $K\in\mathcal{K}_0$ and $q\geq 1,$ the following optimization problem
\begin{equation}\label{1.6}
\inf\bigg\{\int_{S^{n-1}}\Big(\frac{h_L(u)}{h_K(u)}\Big)^qh_K(u)\,dS(K,u): L\in\mathcal{K}_0\ \ \text{and}\ \ |L|=\omega_n\bigg\}
\end{equation}
has a unique solution for $q>1$ and a unique solution up to translation for $q=1$.  

Related to (\ref{1.6}) is the celebrating $L_q$ Minkowski problem: {\em
for $0\neq q\in\mathbb{R}$, under what conditions on a given nonzero finite measure $\mu$ defined on $S^{n-1},$ there is a convex body $K$ (ideally with the origin $o$ in its interior) such that
$h^{q-1}_K d\mu=dS(K,\cdot)$?} 
 The $L_q$ Minkowski problem is a popular problem in geometry and has attracted considerable attention
(see e.g. \cite{Chen2006,CW2006,HMS2004,HLYZ2005,Lutwak1993,EDG2004,VUmanskiy2003,GZhu2015I,GZhu2015II,GZhupreprint}). Solutions to the $L_q$ Minkowski problems have fundamental applications in, for instance, establishing the $L_q$ Sobolev type inequalities (see e.g. \cite{CLYZ2009,HS2009II,EDG2002,GZhang1999}).
 When $q=1$,  the $L_q$ Minkowski problem becomes the classical Minkowski problem.  It has been proved in, e.g. \cite{Minkowski1987, Minkowski1903}, that there exists a unique convex body $K$ (up to a translation) such that $dS(K,\cdot)=d\mu,$ if $\mu$ is not concentrated on any hemisphere of $S^{n-1}$ and has the centroid at the origin. In fact, to solve the classical Minkowski problem is equivalent to find an optimizer of the following optimization problem:
\begin{equation}\label{1.2}
\inf\bigg\{\frac{1}{n}\int_{S^{n-1}}h_L\,d\mu:\ L\in\mathcal{K}_0\ \ \text{and}\ \ |L|=\omega_n\bigg\},
\end{equation}
which can be obtained from  (\ref{1.6}) with $S(K,\cdot)$ replaced by $\mu$ and $q=1$. 
 
 The classical and $L_q$ geominimal surface areas  \cite{Lu96, Petty1974, Petty1985, Ye2015}  are important concepts in convex geometry, which are closely related to (\ref{1.6}). For instance, for $q>0$ and $K\in \mathcal{K}_0$, the $L_q$ geominimal surface area of $K$, denoted by $G_q(K)$, can be formulated by \begin{equation}\label{1.7} G_q(K)=
\inf\bigg\{\int_{S^{n-1}}\Big(\frac{h_L(u)}{h_K(u)}\Big)^qh_K(u)\,dS(K,u): L\in\mathcal{K}_0\ \ \text{and}\ \ |L^\circ|=\omega_n\bigg\},
\end{equation} 
where $L^\circ$ denotes the polar body of $L,$ i.e., $$L^\circ=\{y\in\mathbb{R}^n: \langle x,y \rangle\leq 1\ \text{for any}\ x\in L\}$$ with $\langle \cdot, \cdot\rangle$ the standard inner product on $\mathbb{R}^n$. Note that $|L|=\omega_n$ in (\ref{1.6}) is replaced by $|L^\circ|=\omega_n$ in (\ref{1.7}). However, (\ref{1.6}) and (\ref{1.7}) are completely different, and both of them play fundamental roles in the $L_q$ Brunn-Minkowski theory of convex bodies.   It has been proved \cite{Lu96, Petty1974, ZHY2017} that there exist  convex bodies $T_qK\in\mathcal{K}_0$ for $q>0$, the $L_q$ Petty bodies of $K,$    such that $|(T_qK)^\circ|=\omega_n$ and
$G_q(K)=nV_q(K,T_qK).$ Indeed, a variation problem of (\ref{1.7}),  with  $\mathcal{K}_0$ replaced by $\mathcal{S}_0$ (the set of all star bodies about the origin $o$), defines  the $L_q$ affine surface area of $K$ \cite{Lutwak1991, Lu96, Ye2015}: 
\begin{equation*}
\Omega_q(K)=\inf\bigg\{\int_{S^{n-1}}\Big(\frac{1}{\rho_L(u)h_K(u)}\Big)^qh_K(u)\,dS(K,u): L\in\mathcal{S}_0\ \ \text{and}\ \ |L|=\omega_n\bigg\}
\end{equation*} with $\rho_L: S^{n-1}\rightarrow(0,\infty)$ the radial function of  $L\in\mathcal{S}_0$.  The $L_q$ affine surface area has equivalent convenient integral formulas (see e.g. \cite{Blaschke1923, Hug1996, MW2000,CE1807,CE2004}) and important applications in, such as, the valuation theory, the approximation of convex bodies by polytopes, the $f$-divergence of convex bodies and the $L_q$ affine isoperimetric inequalities (see e.g. \cite{Gruber1993, JW2014,LR1999, LR2010, LSW2006,GE2012,CE2004,Werner2012,Werner2012c,WY2008,GZhang2007}).

Extension from the $L_q$ Brunn-Minkowski theory to the Orlicz theory is rather dedicated. In view of (\ref{1.4})-(\ref{1.7}), a major task is to get the ``right" formula of the Orlicz mixed volume. In the Orlicz theory, there are at least 3 different ways to define the Orlicz mixed volume and each of them has its own advantages. These Orlicz mixed volumes are given as follows: for $K,L\in\mathcal{K}_0,$ and $\phi,\ \varphi:(0,\infty)\rightarrow(0,\infty)$ continuous functions,
\begin{eqnarray}
V_\varphi(K,L)&=&\frac{1}{n}\int_{S^{n-1}}\varphi\Big(\frac{h_L(u)}{h_K(u)}\Big)h_K(u)\,dS(K,u),\label{1.8}\\
V_{\varphi,\phi}(K,L)&=&\frac{1}{n}\int_{S^{n-1}}\frac{\varphi(h_L(u))}{\phi(h_K(u))}\,dS(K,u),\label{1.9}
\end{eqnarray}
and if in addition $\varphi\in\mathcal{I},$
\begin{equation}
\widehat{V}_\varphi(K,L)=\inf\bigg\{\lambda>0:\int_{S^{n-1}}\varphi\Big(\frac{n|K|\cdot h_L(u)}{\lambda \cdot h_K(u)}\Big)h_K(u)\,dS(K,u)\leq n|K|\bigg\},
\label{1.10}\end{equation}
where $\mathcal{I}$ refers to the set of continuous functions $\varphi: (0,\infty)\rightarrow(0,\infty)$ such that
$\varphi$ is strictly increasing, $\lim_{t \to 0^+}\varphi(t)=0,$ $\varphi(1)=1$ and $\lim_{t \to \infty}\varphi(t)=\infty.$
The $L_q$ mixed volume ($q>0$) defined by (\ref{1.4}) is a special case of (\ref{1.8})-(\ref{1.10}), namely, for  $\varphi(t)=t^q$ and $\phi(t)=t^{q-1}$, $$V_q(K,L)=V_\varphi(K,L)=V_{\varphi,\phi}(K,L)=n^{-q}\cdot|K|^{1-q}\big(\widehat{V}_\varphi(K,L)\big)^q.$$

Note that $V_\varphi(\cdot,\cdot)$  and $V_{\varphi,\phi}(\cdot,\cdot)$ have geometric interpretations (see e.g. \cite{GHW2014,DHG2014,ZHY2017}); while there are no geometric interpretations for $\widehat{V}_\varphi(\cdot,\cdot)$ in literature. Also note that $\widehat{V}_{\varphi}(\cdot,\cdot)$ is homogeneous \cite{ZHY2017} but $V_\varphi(\cdot,\cdot)$ and $V_{\varphi,\phi}(\cdot,\cdot)$ are not homogeneous.  For $V_\varphi(\cdot,\cdot)$ and $\widehat{V}_{\varphi}(\cdot,\cdot)$, one has the Orlicz-Minkowski inequalities \cite{GHW2014,DHG2014}: for $K,L\in\mathcal{K}_0$ and $\varphi$ being a convex function, then
\begin{equation}
V_\varphi(K,L)\geq |K|\cdot\varphi\bigg(\bigg(\frac{|L|}{|K|}\bigg)^{\frac{1}{n}}\bigg),\label{1.12}
\end{equation}
and if in addition $\varphi\in\mathcal{I}$,
\begin{equation}
\widehat{V}_{\varphi}(K,L)\geq n|K|^{\frac{n-1}{n}}|L|^{\frac{1}{n}}.\label{1.13}
\end{equation}
Equality holds if and only if $K$ and $L$ are dilates if $\varphi$ is also strictly convex.
The Orlicz-Minkowski inequalities (\ref{1.12}) and (\ref{1.13}) imply that, if $\varphi\in\mathcal{I}$ is strictly convex,
\begin{eqnarray}
\inf\Big\{V_{\varphi}(K,L): L\in\mathcal{K}_0\ \ \text{and}\ \ |L|=\omega_n\Big\} \ \ \mathrm{and}\ \ 
\inf\Big\{\widehat{V}_{\varphi}(K,L): L\in\mathcal{K}_0\ \ \text{and}\ \ |L|=\omega_n\Big\} \label{1-1-1-1} \end{eqnarray}   have unique solutions, respectively. It seems intractable to pose the Minkowski type problems related to $V_\varphi(\cdot,\cdot)$ and $\widehat{V}_\varphi(\cdot,\cdot)$.  However,  the Orlicz-Minkowski problem can be asked  (related to $V_{\varphi,\phi}(\cdot,\cdot)$): {\em under what conditions on a nonzero finite measure $\mu$ on $S^{n-1}$ and a continuous function $\varphi:(0,\infty)\rightarrow(0,\infty)$, there exists a convex body $K$, such that, $\tau \,dS(K,\cdot)=\varphi(h_K)\,d\mu$ for some positive constant $\tau$?}  Solutions of this Orlicz-Minkowski problem can be found in \cite{HLYZ2010,HH2012,Li2014}.

  Ye \cite{Ye2015c} and Zhu, Hong and Ye \cite{ZHY2017} investigated the following optimization problems and gave detailed studies of the (nonhomogeneous and homogeneous) Orlicz geominimal surface areas $G^{orlicz}_\varphi(\cdot)$ and $\widehat{G}^{orlicz}_\varphi(\cdot)$: under certain conditions on $\varphi,$ let
\begin{eqnarray}
G^{orlicz}_\varphi(K)&=&\inf\Big\{nV_{\varphi}(K,L): L\in\mathcal{K}_0\ \ \text{and}\ \ |L^\circ|=\omega_n\Big\}  \label{1-1-1-1---2}\\
\widehat{G}^{orlicz}_\varphi(K)&=&\inf\Big\{\widehat{V}_{\varphi}(K,L): L\in\mathcal{K}_0\ \ \text{and}\ \ |L^\circ|=\omega_n\Big\}.\label{1-1-1-2}
\end{eqnarray}
  In particular, Zhu, Hong and Ye \cite{ZHY2017} proved that, under certain conditions on $\varphi,$ there exist convex bodies $T_\varphi K$ and $\widehat{T}_\varphi K,$ called the Orlicz-Petty bodies of $K,$ such that $|(T_\varphi K)^\circ|=|(\widehat{T}_\varphi K)^\circ|=\omega_n,$  $$G^{orlicz}_\varphi(K)=nV_{\varphi}(K,T_\varphi K)\ \ \text{and}\ \ \ \widehat{G}^{orlicz}_\varphi(K)=\widehat{V}_{\varphi}(K,\widehat{T}_\varphi K).$$
One can also see \cite{SHG2015} for a special case.

Motivated by (\ref{1-1-1-1})-(\ref{1-1-1-2}) and  the relations between (\ref{1.6})-(\ref{1.7}), we propose and study the following problems in Section \ref{section:3}. \vskip 2mm \noindent {\bf The polar Orlicz-Minkowski problems:} {\em under what conditions on a nonzero finite measure $\mu$ and a function $\varphi:(0,\infty)\rightarrow(0,\infty)$, there exists a convex body $K\in\mathcal{K}_0$ such that
$K$ is an optimizer of the following optimization problems:
\begin{equation}\label{Orlicz-Minkowski problem}
\inf/\sup \bigg\{\int_{S^{n-1}}\varphi\big( h_L \big) \,d \mu: L \in \mathcal{K}_{0} \ \text{and}\   |L^\circ|=\omega_{n}\bigg\}.
\end{equation}}\\ \noindent The polar Orlicz-Minkowski problems are not the same as (\ref{1-1-1-1---2})-(\ref{1-1-1-2}) because the measure $\mu$ is not assumed to be related to any convex bodies, and are also  completely different from the Orlicz-Minkowski problems.  

Our main result in Section \ref{section:3} is summarized in the following theorem. Let $\Omega$  be the set of all nonzero finite Borel measures on $S^{n-1}$ that are not concentrated on any hemisphere of $S^{n-1}$.
\begin{theorem}
\noindent Let $\mu\in\Omega$  and $\varphi \in \mathcal{I}.$ Then there exists a convex body $M\in\mathcal{K}_0$ such that $|M^\circ|=\omega_{n}$ and
$$\int_{S^{n-1}}\varphi(h_M(u))d\mu(u)=\inf\bigg\{\int_{S^{n-1}}\varphi(h_L(u))\,d\mu(u):L\in\mathcal{K}_0\ \ \text{and}\ \ |L^\circ|=\omega_n\bigg\}.$$
Moreover, if $\varphi\in\mathcal{I}$ is convex, then $M$ is the unique solution to the polar Orlicz-Minkowski problem.
\end{theorem}

In Section \ref{section:4}, we replace $V_{\varphi}(\cdot,\cdot)$ and $\widehat{V}_{\varphi}(\cdot,\cdot)$ by their $p$-capacitary counterparts, and study the existence and continuity of the $p$-capacitary Orlicz-Petty bodies.
Hereafter,  for $K,L\in\mathcal{K}_0,$ $p\in(1,n)$ and $\varphi:(0,\infty)\rightarrow(0,\infty),$ the Orlicz mixed $p$-capacities of $K$ and $L$ are given by:
\begin{eqnarray}
 C_{p,\varphi}(K,L) &=&\frac{p-1}{n-p}\int_{S^{n-1}}\varphi\left( \frac{h_L(u)}{h_K(u)} \right)h_K(u)\,d \mu_p(K,u),\label{1.18}\\
 1&=&\int_{S^{n-1}}\varphi\bigg(\frac{C_p(K)\cdot h_{L}(u)}{\widehat{C}_{p,\varphi}(K,L)\cdot h_{K}(u)}\bigg)\,d \mu^*_p(K,u)\ \ \ \  \text{if in addition}\  \varphi\in\mathcal{I},\label{1.19}
\end{eqnarray}
where $\mu_p(K, \cdot)$ and $\mu^*_p(K,\cdot)$ are measures  given by (\ref{measure-38}) and (\ref{measure-37}), respectively.  Related to $C_{p,\varphi}(\cdot,\cdot)$ and $\widehat{C}_{p,\varphi}(\cdot,\cdot)$, there are the $p$-capacitary Orlicz-Minkowski inequalities: for $p\in(1,n),$ $K,L\in\mathcal{K}_0$ and $\varphi\in\mathcal{I}$ being convex \cite{HYZ2017I}, one has,
\begin{eqnarray*}
C_{p,\varphi}(K,L)\geq C_p(K)\cdot\varphi\bigg(\bigg(\frac{C_p(L)}{C_p(K)}\bigg)^{\frac{1}{n-p}}\bigg)\ \ \  \text{and}\ \ \ 
\widehat{C}_{p,\varphi}(K,L)\geq C_p(K)\bigg(\frac{C_p(L)}{C_p(K)}\bigg)^{\frac{1}{n-p}},
\end{eqnarray*}
with equality if and only if $K$ and $L$ are dilates if $\varphi$ is strictly convex. One can also ask the $p$-capacitary $L_q$ and Orlicz Minkowski problems (i.e., with $S(K,\cdot)$ replaced by $\mu_p(K,\cdot)$); these problems have received extensive attention (see e.g. \cite{Cole2015, HYZ2017I,Jenkinson1996,Jenkinson1996c,DG}). 

Our major interest in  Section \ref{section:4} is the following problems: for $p\in(1,n),$ find optimizers of the following optimization problems:
\begin{eqnarray*}
& &\sup/\inf \Big\{C_{p,\varphi}(K,L):L\in\mathcal{K}_0 \ \ \text{and}\  \ |L^\circ|=\omega_{n}\Big\},\\ 
& &\sup/\inf \Big\{\widehat{C}_{p,\varphi}(K,L):L\in\mathcal{K}_0 \ \ \text{and}\ \ |L^\circ|=\omega_{n}\Big\}.
\end{eqnarray*}
The main result in Section \ref{section:4} is stated as follows. 
\begin{theorem}\label{Theorem existence of petty bodies'}
\noindent Let $K \in \mathcal{K}_0$ be a convex body and $\varphi \in \mathcal{I}.$

\noindent(i) There exists a convex body $M\in\mathcal{K}_0$ such that $|M^\circ|=\omega_{n}$ and
$$C_{p,\varphi}(K,M)=\inf \Big\{C_{p,\varphi}(K,L):L\in\mathcal{K}_0 \ \ \mbox{and}\ \ |L^\circ|=\omega_{n}\Big\}.$$
\noindent(ii) There exists a convex body $\widehat{M}\in\mathcal{K}_0$ such that $|\widehat{M}^\circ|=\omega_{n}$ and
$$ \widehat{C}_{p,\varphi}(K,\widehat{M})=\inf \Big\{\widehat{C}_{p,\varphi}(K,L):L\in\mathcal{K}_0 \ \ \mbox{and}\ \ |L^\circ|=\omega_{n}\Big\}.$$
In addition, if $\varphi\in\mathcal{I}$ is convex, then both $M$ and $\widehat{M}$ are unique.
\end{theorem}
The convex bodies $M$ and $\widehat{M}$ in Theorem \ref{Theorem existence of petty bodies'} are called the $p$-capacitary Orlicz-Petty bodies of $K$.
The continuity of the $p$-capacitary Orlicz-Petty bodies is provided in Theorem \ref{Theorem 4.2}.

\section{ Preliminaries and Notations}\label{section:2}
A subset $K\subseteq \mathbb{R}^n$ is said to be convex if $\lambda x+(1-\lambda) y \in K$ for any $\lambda \in [0,1]$ and $x,y\in K$. A convex body is a convex compact subset of $\mathbb{R}^n$ with nonempty interior. A convex body $K$ is said to be origin-symmetric if $-x\in K$ for any $x\in K$. We use $\mathcal{K}$ and $\mathcal{K}_0\subseteq \mathcal{K}$ to denote the set of all convex bodies, respectively,  and the set of all convex bodies with the origin in their interiors. The Minkowski sum of $K,L\in\mathcal{K}$, denoted by $K+L$, is defined by $$K+L=\{x+y: x\in K, y\in L\}.$$
For $\lambda \in \mathbb{R}$, the scalar product of $\lambda$ and $K,$ denoted by $\lambda K,$ is defined by $\lambda K=\{\lambda x: x\in K\}.$
For $K \in \mathcal{K},$ $|K|$ refers to the volume of $K$. In particular, $\omega_n$ represents the volume of the unit ball $B_2^n \subseteq \mathbb{R}^n.$ For $K \in \mathcal{K},$ one can define the volume radius of $K$ by $$\text{vrad}(K)=\bigg(\frac{|K|}{\omega_n}\bigg)^{\frac{1}{n}}.$$

The support function of a nonempty convex compact set $K\subseteq\mathbb{R}^n$, $h_K: S^{n-1}\rightarrow \mathbb{R}$, is defined by
$$h_K(u)=\max_{x\in K}{\langle x,u\rangle} \ \ \text{for any} \ u \in S^{n-1},$$
where $S^{n-1}$ is the unit sphere. It can be easily checked that, for any $\lambda \geq 0$ and $K,L \in \mathcal{K},$ $h_{\lambda K}(u)=\lambda h_K(u)$ and $h_{K+L}(u)=h_K(u)+h_L(u) \ \text{for any} \ u \in S^{n-1}.$ The Hausdorff distance between two convex compact sets $K, L\subseteq \mathbb{R}^n$,
denoted by $d_H(K,L)$, is given by
$$
d_H(K,L)=\max_{u\in S^{n-1}}|h_K(u)-h_L(u)|.$$ For a sequence ${\left\{K_i\right\}}_{i=1}^{\infty}\subseteq \mathcal{K}_0$ and a convex compact set $K$, $K_i \rightarrow K$ as $i \rightarrow \infty$ with respect to the Hausdorff metric means that $d_H(K_i,K) \rightarrow 0$ as $i \rightarrow \infty$.

A subset $L\subseteq\mathbb{R}^{n}$ is called a star-shaped set about the origin if for $\text{any}\ x\in L$, the line segment from the origin $o$ to $x$ is contained in $L$. The radial function of a star-shaped set $L$ about the origin $o$, $\rho_L:S^{n-1}\rightarrow [0,\infty)$, is defined by
$$\rho_L(u)=\max\{r\geq0:\,\,r u\in L\} \ \ \  \text{for any} \ u \in S^{n-1}.$$
A star-shaped set $L\subseteq\mathbb{R}^n$ about the origin $o$ is called a star body about the origin $o$ if the radial function $\rho_L$ is positive and continuous on $S^{n-1}.$ Denote by $\mathcal{S}_0$ the set of all star bodies about the origin $o$. Obviously, $\mathcal{K}_0\subseteq\mathcal{S}_0$.

The polar body $K^{\circ}$ of $K \in \mathcal{K}_0$ is defined by $$K^{\circ}=\{x\in\mathbb{R}^n:\langle x,y\rangle \leq 1 \ \text{for any} \ y \in K\},$$ where $\langle \cdot, \cdot \rangle$ is the standard inner product in $\mathbb{R}^n$. By $K^{\circ\circ}$, we mean the polar body of $K^{\circ}$, and hence $K^{\circ\circ}=K$ if $K\in \mathcal{K}_0$ \cite[Theorem 1.6.1]{Schn2014}. 
 It can be proved \cite{Schn2014} that for any $K \in \mathcal{K}_0,$
\be\label{measure-3}
\rho_{K^{\circ}}(u)=\frac{1}{h_K(u)}\ \ \,\text{and}\ \ \ h_{K^{\circ}}(u)=\frac{1}{\rho_K(u)} \,\, \, \, \text{for any}\ u \in S^{n-1}.
\ee For $K\in \mathcal{S}_0$, the following volume formula for $|K|$ holds:
\be\label{measure-4}|K|
=\frac{1}{n}\int_{S^{n-1}}\rho_K(u)^n \,d\sigma(u),
\ee
where $\sigma(\cdot)$ is the spherical measure on $S^{n-1}.$ Associated to each $K\in\mathcal{K}_0$, the surface area measure of $K$ on $S^{n-1}$, denoted by $S(K,\cdot),$ is defined by: for any measurable subset $A\subseteq S^{n-1}$,
$$
S(K,A)=\int_{\nu_K^{-1}(A)} \,d\mathcal{H}^{n-1},$$
where $\nu_K^{-1}:S^{n-1}\rightarrow \partial K$ is the inverse Gauss map of $K$ 
and $\mathcal{H}^{n-1}$ is the $(n-1)$-dimensional Hausdorff measure on $\partial K.$

Denote by $C(S^{n-1})$ the set of all continuous functions on $S^{n-1}.$ Let ${\left\{\mu_i\right\}}_{i=1}^{\infty}$ and $\mu$  be measures on $S^{n-1}$. We say $\mu_i\rightarrow \mu$ weakly if for any $f\in C(S^{n-1})$,
$$\lim_{i\rightarrow \infty}\int_{S^{n-1}}f\,d\mu_i=\int_{S^{n-1}}f\,d\mu.$$

The following lemma regarding the weak convergence $\mu_i\rightarrow\mu$ is often used.
\begin{lemma}\label{uniformly converge-lemma}
If a sequence of measures ${\left\{\mu_i\right\}}_{i=1}^{\infty}$ on $S^{n-1}$ converges weakly to a finite measure $\mu$ on $S^{n-1}$ and a sequence of functions ${\left\{f_i\right\}}_{i=1}^{\infty}\subseteq C(S^{n-1})$ converges uniformly to a function $f\in C(S^{n-1})$, then
$$\lim_{i\rightarrow \infty}\int_{S^{n-1}}f_i\,d\mu_i=\int_{S^{n-1}}f\,d\mu.$$
\end{lemma}

We shall also need the following lemma. 

\begin{lemma}  \label{blaschke-selection}
Let ${\left\{K_i\right\}}_{i=1}^{\infty}\subseteq \mathcal{K}_0$ be a uniformly bounded sequence such that the sequence ${\left\{|K_i^{\circ}|\right\}}_{i=1}^{\infty}$ is bounded. Then, there exists a subsequence $\{K_{i_j}\}_{j=1}^{\infty}$ of ${\left\{K_i\right\}}_{i=1}^{\infty}$ and a convex body $K\in   \mathcal{K}_0$ such that $K_{i_j}\rightarrow K$.  Moreover, if $|K_i^{\circ}|=\omega_n$ for all $i=1, 2, \cdots$, then $|K^\circ|=\omega_n$.   \end{lemma}
\begin{proof}   Lutwak \cite{Lu96} proved that,  if  ${\left\{K_i\right\}}_{i=1}^{\infty}\subseteq \mathcal{K}_0$ and $K$ a convex compact set satisfy that $K_i\rightarrow K$ as $i\rightarrow\infty$ with respect to the Hausdorff metric and the sequence ${\left\{|K_i^{\circ}|\right\}}_{i=1}^{\infty}$ is bounded, then $K\in \mathcal{K}_0$. On the other hand, the Blaschke selection theorem \cite{Schn2014} asserts that, if ${\left\{K_i\right\}}_{i=1}^{\infty}$ is a bounded sequence of convex compact sets in $\mathbb{R}^n$, then there exist a subsequence $\{K_{i_j}\}_{j=1}^\infty$ of $\{K_i\}_{i=1}^\infty$ and a convex compact set $K$, such that, $K_{i_j}\rightarrow K$ as $j\rightarrow \infty$ with respect to the Hausdorff metric. Combining the above two statements, one can immediately get $K_{i_j}\rightarrow K\in \mathcal{K}_0.$ Moreover, if $|K_i^{\circ}|=\omega_n$ for all $i=1, 2, \cdots$, then $|K^\circ|=\lim_{j\rightarrow \infty} |K_{i_j}^\circ|=\omega_n$, due to the continuity of volume under the Hasdorff metric. 
\end{proof}

Now we provide some basic background for the $p$-capacity, and all results can be found in e.g. \cite{Cole2015, Cole2003,EG2008}. 
Let $C^\infty_c(\mathbb{R}^{n})$ be the set of all infinitely differentiable functions on $\mathbb{R}^n$ with compact supports. For $x\in\mathbb{R}^n$, we use $|x|$ to mean the Euclidean norm of $x.$ For a compact subset $E\subseteq\mathbb{R}^n$ and $1<p<n$, define $C_p(E)$, the $p$-capacity of $E$, by
$$
C_p(E)=\inf\left\{\int_{\mathbb{R}^{n}}|\nabla f(x)|^p\, dx:  \,f\in C^\infty_c(\mathbb{R}^{n}) \ \text{and} \ f(x)\geq 1 \ \text{on} \ x\in E\right\},$$ where $\nabla f$ denotes the gradient of $f$. 
  Clearly,  for two compact sets $E\subseteq F\subseteq \mathbb{R}^n,$ one has $C_p(E)\leq C_p(F).$ Moreover, for any compact set $E\subseteq \mathbb{R}^n$, $C_p(\lambda E)=\lambda^{n-p}C_p(E)$ for any  $\lambda>0,$ and $C_p(E+x_0)=C_p(E)$ for any $x_0 \in \mathbb{R}^{n}$.  Besides, the functional $C_p(\cdot)$ is continuous on $\mathcal{K}_0$ with respect to the Hausdorff metric.

 By the $p$-Laplace equation for $p\in (1, n)$, we mean the equation $\text{div}(|\nabla U|^{p-2}\nabla U)=0,$ with $\text{div}(\cdot)$ the divergence of vector fields in $\mathbb{R}^n$.  The $p$-Laplace equation has a fundamental solution $U_0(x)=|x|^{\frac{p-n}{p-1}}(x\neq o).$
For $K\in \mathcal{K}_0$, the $p$-capacitary function of $K$ is a weak solution of the following $p$-Laplace equation with the boundary conditions:
\begin{equation}
\begin{cases}
\text{div}(|\nabla U|^{p-2}\nabla U)=0,\  \ \ \ \text{in} \ \ \ \mathbb{R}^{n}\setminus K,\\  U(x)=1, \qquad\qquad\quad \ \ \ \ \ \text{on} \ \ \ \partial K,\\ \lim_{| x | \rightarrow \infty}U(x)=0.
\end{cases}\label{formula 2.9}
\end{equation}
It has been proved that there exists a unique solution $U_K$ to (\ref{formula 2.9}) such that
$$
C_p(K)=\int_{\mathbb{R}^{n}\setminus K}|\nabla U_K(x)|^p \,dx
$$
and $U_K\in C(\mathbb{R}^{n}\setminus \text{int}K)\cap C^\infty (\mathbb{R}^{n}\setminus K)$ (see more details in \cite{Cole2003}).

Let $K\in\mathcal{K}_0.$ Define $\mu_p(K,\cdot)$, the $p$-capacitary measure of $K$ on $S^{n-1}$, by (see e.g. \cite[(1.11)]{Cole2015})
\be\label{measure-38}
\mu_p(K,A)=\int_{\nu_K^{-1}(A)}|\nabla U_K(x)|^p \,d\mathcal{H}^{n-1}(x) \ \  \text{ for any measurable subset}\, A\subseteq S^{n-1},\ee
where $U_K$ is the $p$-capacitary function of $K$.
For any $\lambda>0$, one can easily get
$$\mu_p(\lambda K,\cdot)=\lambda^{n-p-1}\mu_p(K,\cdot)\ \ \text{on}\ S^{n-1}.$$  
The translation invariance of $C_p(\cdot)$ yields that for any $K\in\mathcal{K}_0$, the centroid of $\mu_p(K,\cdot)$ is at the origin, i.e.,
$$\int_{S^{n-1}}u d\mu_p(K,u)=o.$$ Moreover, $\mu_p(K,\cdot)$ is not concentrated on any hemisphere of $S^{n-1}$ \cite[Theorem 1]{XZ}, i.e.,
\be\nonumber
\int_{S^{n-1}}\langle v,u \rangle_+\, d\mu_p(K,u)>0\,\,\,\,\text{for any} \ v\in S^{n-1},
\ee where $a_+=\max\{a, 0\}$ for all $a\in \mathbb{R}$. 
The $p$-capacity of $K\in\mathcal{K}_0$ can be calculated by the famous Poincar\'{e} formula (see e.g. \cite{Cole2015}):
\be\label{measure-8}
C_p(K)=\displaystyle{\frac{p-1}{n-p}}\int_{S^{n-1}}h_K(u)\, d\mu_p(K,u).
\ee In particular,  for any $p\in(1,n)$,
\be \nonumber C_p(B^n_2)=\bigg(\displaystyle{\frac{n-p}{p-1}}\bigg)^{p-1}\cdot n\omega_n.\ee
From (\ref{measure-8}), for any $K\in\mathcal{K}_0$, one can define a probability measure $\mu^*_p(K,\cdot)$ on $S^{n-1}$ by 
\be\label{measure-37}
\,d\mu^*_p(K, u)= \bigg(\frac{p-1}{n-p}\bigg)  \cdot \bigg( \frac{h_K(u)}{C_p(K)}\bigg)\cdot \,d\mu_p(K,u) \ \text{for all} \ u\in S^{n-1}.
\ee According to \cite[(8.9)]{Maz2003},  one has the following isocapacitary inequality:
\be\label{measure-36}C_p(K)\geq n {\omega_n}^{p/n}\bigg(\frac{n-p}{p-1}\bigg)^{p-1} |K|^{(n-p)/n}\ee
holds for any $K\in\mathcal{K}_0$.  

The $p$-capacity and the volume belong to a large family of functionals defined on $\mathcal{K}_0$. Such a family of functionals will be called the variational functionals compatible with the mixed volume \cite{HYZ2017} (see also \cite{Cole2005,Cole2010}). We summarize its definition below.
\begin{definition}\label{definition 2.1}
A variational functional $\mathcal{V}:\mathcal{K}_0\rightarrow(0,\infty)$ is said to be compatible with the mixed volume if $\mathcal{V}$ satisfies the following properties: 

\noindent (i) homogeneous, i.e., there exists a constant $\alpha\neq 0$ such that $\mathcal{V}(\lambda K)=\lambda^{\alpha}\mathcal{V}(K)$ for any $\lambda>0$ and any $K\in\mathcal{K}_0$;

\noindent (ii) translation invariant, i.e., $\mathcal{V}(K+x)=\mathcal{V}(K)$ for any $x\in\mathbb{R}^n$ and any $K\in\mathcal{K}_0$;

\noindent (iii) monotone increasing, i.e., $\mathcal{V}^{\frac{1}{\alpha}}(K)\leq\mathcal{V}^{\frac{1}{\alpha}}(L)$ if $K\subseteq L$;

\noindent (iv) the Brunn-Minkowski inequality, i.e., for any $\lambda\in[0,1]$ and $K,L\in\mathcal{K}_0$,
$$\mathcal{V}^{\frac{1}{\alpha}}(\lambda K+(1-\lambda)L)\geq\lambda\mathcal{V}^{\frac{1}{\alpha}}(K)+(1-\lambda)\mathcal{V}^{\frac{1}{\alpha}}(L)$$
with equality if and only if $K$ and $L$ are homothetic to each other;

\noindent (v) there exists a nonzero finite measure $S_\mathcal{V}(K,\cdot)$ on $S^{n-1}$,  such that, for any $L\in\mathcal{K}_0$,
$$\displaystyle{\frac{1}{\alpha}}\cdot\lim_{\epsilon\rightarrow 0^+}\frac{\mathcal{V}(K+\epsilon L)-\mathcal{V}(K)}{\epsilon}=\int_{S^{n-1}}h_L(u)\,dS_\mathcal{V}(K,u),$$
and $S_\mathcal{V}(K,\cdot)$ is not concentrated on any hemisphere of $S^{n-1}$. Moreover, the convergence of $K_i\rightarrow K$ with respect to the Hausdorff metric implies that $S_\mathcal{V}(K_i,\cdot)$ converges weakly to $S_\mathcal{V}(K,\cdot)$.

\end{definition}
Besides the $p$-capacity and the volume, there are many other functionals on $\mathcal{K}_0$ satisfying the conditions in Definition \ref{definition 2.1}, such as $\tau(K),$ the torsional rigidity of $K$, whose definition is given by (see e.g. \cite{Cole2005}):
$$\frac{1}{\tau(K)}=\inf\bigg\{\displaystyle{\frac{\int_{K}|\nabla u(x)|^2\,dx}{(\int_{K}|u(x)|\,dx)^2}}: \ \  u\in W^{1,2}_0(\text{int}K)\ \text{and}\ \int_{K}|u(x)|\,dx>0\bigg\},$$
where $W^{1,2}(\text{int}K)$ refers to the Sobolev space of the functions in $L^2(\text{int}K)$ whose first order weak derivatives belong to $L^2(\text{int}K),$ and $W^{1,2}_0(\text{int}K)$ denotes the closure of $C^\infty_c(\text{int}K)$ in  the Sobolev space $W^{1,2}(\text{int}K).$
By the definition of torsional rigidity, one can easily get that $\tau(K)\leq\tau(L)$ if $K\subseteq L.$ Moreover, for any $K\in\mathcal{K}_0,$
$$
\tau(\lambda K)=\lambda^{n+2}\tau(K) \ \ \text{for any} \ \lambda>0 \ \ \text{and}\ \ \tau( K+x)=\tau(K) \ \ \text{for any} \ x\in\mathbb{R}^n.
$$
The torsional rigidity satisfies the Brunn-Minkowski inequality, i.e., for any $\lambda\in[0,1]$ and $K,L\in\mathcal{K}_0$,
$$\tau^{\frac{1}{n+2}}(\lambda K+(1-\lambda)L)\geq\lambda\tau^{\frac{1}{n+2}}(K)+(1-\lambda)\tau^{\frac{1}{n+2}}(L)$$
with equality if and only if $K$ and $L$ are homothetic to each other; please refer to \cite{Cole2005} for more details. For any $K\in\mathcal{K}_0,$ there exists a unique solution $U_{\tau, K}\in C^\infty(\text{int}K) \cap C(K)$ to the following boundary value equation:
\begin{equation*}
\begin{cases}
\text{div}(\nabla U)=-2\  \,\,\,\text{in} \,\,\text{int}K,\\U(x)=0 \qquad\ \ \ \ \text{on} \,\,\partial K.
\end{cases}
\end{equation*} 
One can define $\mu_{\tau}(K,\cdot),$ a nonnegative Borel measure on $S^{n-1},$ as follows (see e.g. \cite{Cole2010}): for any measurable subset $A\subseteq S^{n-1}$,
$$
\mu_{\tau}(K,A)=\int_{\nu_K^{-1}(A)}|\nabla U_{\tau, K}(x)|^2 \,d\mathcal{H}^{n-1}(x).
$$
It follows from \cite[Corollary 1]{Cole2010} and \cite[Theorem 6]{Cole2010} that the measure $\mu_{\tau}(K,\cdot)$ satisfies condition (v) in Definition \ref{definition 2.1}.
\section{The polar Orlicz-Minkowski problems}\label{section:3}

 In this section, we prove that the polar Orlicz-Minkowski problem (\ref{Orlicz-Minkowski problem}) is solvable under the assumptions $\varphi\in\mathcal{I}$ and $\mu\in\Omega.$ Recall that $\Omega$  is the set of all nonzero finite Borel measures on $S^{n-1}$ that are not concentrated on any hemisphere of $S^{n-1}$, namely, for any $\mu\in \Omega$ and for any $v\in S^{n-1}$, one has $$\int_{S^{n-1}}\langle u, v\rangle _+ \,d\mu(u)>0.$$  The continuity of the solutions to  the polar Orlicz-Minkowski problems is also provided. For convenience, if $\varphi\in\mathcal{I}$, let
\begin{equation}\label{polar Orlicz-Minkowski problem 3.2}
\widehat{\mathcal{G}}_{\varphi}(\mu)=\inf \bigg\{\int_{S^{n-1}}\varphi\big( h_L \big) \,d \mu: L \in \mathcal{K}_{0} \ \text{and}\   |L^\circ|=\omega_{n}\bigg\}.
\end{equation}
Clearly, as $\varphi(1)=1$,
$$\widehat{\mathcal{G}}_{\varphi}(\mu)\leq \int_{S^{n-1}}\varphi\big( h_{B^n_2} \big) \,d \mu\leq\mu(S^{n-1})<\infty.$$
Moreover, due to $|(\text{vrad}(L^\circ)L)^\circ|=\Big|\displaystyle{\frac{L^\circ}{\text{vrad}(L^\circ)}}\Big|=\omega_n$ and $h_{\text{vrad}(L^\circ)L}=\text{vrad}(L^\circ)h_L$ for any $L\in\mathcal{K}_0,$ one has
$$\widehat{\mathcal{G}}_{\varphi}(\mu)=\inf_{L\in \mathcal{K}_{0}}\bigg\{\int_{S^{n-1}}\varphi\big( \text{vrad}(L^\circ)h_L \big) \,d \mu\bigg\}.$$

\begin{theorem} \label{Theorem3.1}
\noindent Let $\mu\in\Omega$  and $\varphi \in \mathcal{I}$. Then there exists a convex body $M\in\mathcal{K}_0$ such that $|M^\circ|=\omega_{n}$ and
$$\widehat{\mathcal{G}}_{\varphi}(\mu)=\int_{S^{n-1}}\varphi\big( h_M \big) \,d \mu .$$
Moreover, if $\varphi\in\mathcal{I}$ is convex, then $M$ is the unique solution to the polar Orlicz-Minkowski problem (\ref{polar Orlicz-Minkowski problem 3.2}).
\end{theorem}
\begin{proof} The fact that $\mu$ is not concentrated on any hemisphere of $S^{n-1},$ together with the monotone convergence theorem, implies that, for any given $v\in S^{n-1}$, 
$$\lim_{j\rightarrow \infty}\int_{\{u\in S^{n-1}:\,\langle u,v \rangle_+\geq\frac{1}{j}\}}\langle u,v \rangle_+ \,d \mu(u)=\int_{S^{n-1}}\langle u,v \rangle_+\,d \mu(u)>0.$$
Thus, there exists an integer $j_0\geq1$ such that
\begin{equation}\label{3.2'}
\int_{\{u\in S^{n-1}:\,\langle u,v \rangle_+\geq\frac{1}{j_0}\}}\langle u,v \rangle_+ \,d \mu(u)>0.\end{equation}

Let $\{M_i\}_{i=1}^\infty\subseteq \mathcal{K}_0$ be a sequence of convex bodies such that $|M^\circ_i|=\omega_{n}$ for any $i\geq 1$ and
\be \label{measure-34}
\int_{S^{n-1}}\varphi\big( h_{M_i}\big) \,d \mu\rightarrow \widehat{\mathcal{G}}_{\varphi}(\mu)<\infty.
\ee Let $R_i=\rho_{M_i}(u_i)=\max_{u\in S^{n-1}}\{\rho_{M_i}(u)\}$. Obviously, $
h_{M_i}(u) \geq R_i \cdot \langle u,u_i \rangle_+ $ for any $u \in S^{n-1}
$ and any $i\geq1.$ Since $S^{n-1}$ is compact, we can assume $u_i\rightarrow v\in S^{n-1}$ as $i\rightarrow \infty$. 
We now prove $\sup_{i\geq 1} R_i<\infty$. This will follow if we can get a contradiction by assuming $R_{i} \rightarrow \infty$ as $i\rightarrow \infty$ (or, more precisely, some subsequence $R_{i_j}\rightarrow\infty$ as $j\rightarrow \infty$). To this end, by the monotonicity of $\varphi,$ Fatou's Lemma,  and  $h_{M_i}(u) \geq R_i \cdot \langle u,u_i \rangle_+$, one has, for any positive constant $C>0$,
\begin{eqnarray*}
\nonumber \widehat{\mathcal{G}}_{\varphi}(\mu)&=& \lim_{i\rightarrow\infty}\int_{S^{n-1}}\varphi\big( h_{M_i} \big) \,d \mu\\ \nonumber
&\geq& \liminf_{i\rightarrow \infty}\int_{S^{n-1}}\varphi\big( R_i \cdot \langle u,u_i \rangle_+ \big) \,d \mu(u)\\ \nonumber
&\geq& \liminf_{i\rightarrow \infty}\int_{S^{n-1}}\varphi\big( C\cdot \langle u,u_i \rangle_+ \big) \,d \mu(u)\\ \nonumber
&\geq& \int_{S^{n-1}}\liminf_{i\rightarrow \infty}\varphi\big( C\cdot \langle u,u_i \rangle_+ \big) \,d \mu(u)\\ \nonumber
&=&\int_{S^{n-1}}\varphi\big( C\cdot \langle u,v \rangle_+ \big) \,d \mu(u)\\
&\geq& \varphi\left( \displaystyle{\frac{C}{j_0}} \right)\,\cdot\int_{\{u\in S^{n-1}:\,\langle u,v \rangle_+\geq\frac{1}{j_0}\}}\langle u,v \rangle_+ \,d\mu(u).\nonumber
\end{eqnarray*}
Letting $C\rightarrow \infty$, it follows from (\ref{3.2'}) and $\varphi\in \mathcal{I}$ that $\widehat{\mathcal{G}}_{\varphi}(\mu)\geq\infty$, which is impossible. Hence $\sup_{i\geq 1} R_i<\infty$ and $\{M_i\}_{i=1}^\infty$ is bounded.

By Lemma \ref{blaschke-selection}, there exists a subsequence of $\{M_i\}_{i=1}^\infty$ which converges to some convex body $M\in\mathcal{K}_0$ with $|M^\circ|=\omega_{n}.$ Without loss of generality, we assume $M_i\rightarrow M$ as $i\rightarrow \infty$. Thus there exist two positive constants $r_0$ and $R_0$ such that for any $i\geq1$ and any $u\in S^{n-1}$,
$$r_0\leq h_{M_i}(u),h_M(u)\leq R_0.$$ From the fact that $\varphi$ is continuous on $[r_0,R_0]$ and the dominated convergence theorem, one has
$$\int_{S^{n-1}}\varphi\big( h_{M_i} \big) \,d \mu\rightarrow \int_{S^{n-1}}\varphi\big( h_{M} \big) \,d \mu.$$
Together with (\ref{measure-34}), one has
$$\widehat{\mathcal{G}}_{\varphi}(\mu)=\int_{S^{n-1}}\varphi\big( h_{M} \big) \,d \mu.$$
Hence $M$ is a solution to the polar Orilicz-Minkowski problem (\ref{polar Orlicz-Minkowski problem 3.2}).

For the uniqueness, let $M_1$ and $M_2$ be two convex bodies such that $|{M^\circ_1}|=|{M^\circ_2}|=\omega_{n}$ and
$$\widehat{\mathcal{G}}_{\varphi}(\mu)=\int_{S^{n-1}}\varphi\big( h_{M_1} \big) \,d \mu=\int_{S^{n-1}}\varphi\big( h_{M_2} \big) \,d \mu .$$
Let $M_{0}=\displaystyle{\frac{M_1+M_2}{2}}.$ Clearly, due to the fact that $t^{-n}$ is strictly convex, (\ref{measure-3}) and (\ref{measure-4}), $\text{vrad}(M^\circ_{0})\leq 1$ with $\text{vrad}(M^\circ_{0})=1$ if and only if $M_1=M_2.$ By the strict monotonicity of $\varphi$ and the fact that $\varphi$ is convex, one has
\begin{eqnarray*}
\widehat{\mathcal{G}}_{\varphi}(\mu)&\leq&\int_{S^{n-1}}\varphi\big(\text{vrad}(M^\circ_{0})\cdot h_{M_{0}} \big) \,d \mu\\&\leq& \int_{S^{n-1}}\varphi\big( h_{M_{0}} \big) \,d \mu\\&=&\int_{S^{n-1}}\varphi\bigg( \frac{h_{M_1}+h_{M_2}}{2} \bigg) \,d \mu\\&\leq& \int_{S^{n-1}} \frac{\varphi\big(h_{M_1}\big)+\varphi\big(h_{M_2}\big)}{2} \,d \mu\\&=&\widehat{\mathcal{G}}_{\varphi}(\mu).
\end{eqnarray*}
This implies $\text{vrad}(M^\circ_{0})=1$ and hence $M_1=M_2$.
\end{proof}

The following proposition states that the solutions to the polar Orlicz-Minkowski problem (\ref{polar Orlicz-Minkowski problem 3.2}) for discrete measures must be polytopes.
\begin{proposition} \label{proposition3.1}
Let $\varphi\in\mathcal{I}$ and $\mu\in \Omega$ be a discrete measure on $S^{n-1}$ whose support $\{u_1,u_2,\cdots, u_m\}\subseteq S^{n-1}$ is not concentrated on any hemisphere of  $S^{n-1}.$ If $M\in\mathcal{K}_0$ is a solution of the polar Orlicz-Minkowski problem (\ref{polar Orlicz-Minkowski problem 3.2}) for $\mu$, then $M$ is a polytope with $u_1,u_2,\cdots, u_m$ being the unit normal vectors of its faces.
\end{proposition}
\begin{proof}
Let $P$ be a polytope  with $u_1,u_2,\cdots, u_m$ being the unit normal vectors of its faces such that \begin{equation}\label{cir-12-22} P=\bigcap_{1\leq i\leq m} \big\{x\in \mathbb{R}^n: \langle x, u_i\rangle\leq h_M(u_i)\big\}. \end{equation}  Thus, $h_P(u_i)=h_M(u_i)$ $(1\leq i \leq m),$ $P^\circ \subseteq M^\circ$ (as $M\subseteq P$) and $\text{vrad}(P^\circ)\leq \text{vrad}(M^\circ)=1$.
It follows from the fact that $\varphi$ is strictly increasing that
\begin{eqnarray*}
\inf_{L\in \mathcal{K}_{0}}\bigg\{\int_{S^{n-1}}\varphi\big( \text{vrad}(L^\circ)h_{L} \big) \,d\mu\bigg\}&\leq & \int_{S^{n-1}}\varphi\big( \text{vrad}(P^\circ)h_{P} \big)\, d\mu\\ &\leq& \int_{S^{n-1}}\varphi\big( h_{P} \big)\,d\mu\\
&=&\sum^m_{i=1}\varphi\left( h_{P}(u_i)\right)\cdot\mu(\{u_i\})\\
&=&\sum^m_{i=1}\varphi\left( h_{M}(u_i)\right)\cdot\mu(\{u_i\})\\
&=&\int_{S^{n-1}}\varphi\big( h_{M} \big)\,d\mu\\&=&\inf_{L\in \mathcal{K}_{0}}\bigg\{\int_{S^{n-1}}\varphi\big(\text{vrad}(L^\circ) h_{L} \big) \,d\mu\bigg\}.
\end{eqnarray*}
This shows $\text{vrad}(P^\circ)=\text{vrad}(M^\circ)=1$ and hence $M=P$.
\end{proof}
Note that if $\varphi\in\mathcal{I}$ is not convex, then the solutions to the polar Orlicz-Minkowski problem (\ref{polar Orlicz-Minkowski problem 3.2}) may not be unique. We use $\mathcal{M}_{\varphi}(\mu)$ for the set of all convex bodies satisfying the polar Orlicz-Minkowski problem (\ref{polar Orlicz-Minkowski problem 3.2}) for $\mu\in\Omega$. When $\varphi\in\mathcal{I}$ is convex, $\mathcal{M}_\varphi(\mu)$ contains only one convex body, and hence $\mathcal{M}_\varphi(\cdot)$ defines an operator on $\Omega$.

The following theorem states the continuity of $\widehat{\mathcal{G}}_{\varphi}(\cdot)$ and $\mathcal{M}_\varphi(\cdot)$. 
\begin{theorem}\label{Theorem 3.2}
Let $\{\mu_i\}_{i=1}^\infty\subseteq \Omega$ and $\mu\in\Omega$ be such that $\mu_i$ converges weakly to $\mu$ as $i\rightarrow \infty$.

\noindent(i) If $\varphi \in \mathcal{I}$, then $\widehat{\mathcal{G}}_{\varphi}(\mu_i)\rightarrow\widehat{\mathcal{G}}_{\varphi}(\mu)$.

\noindent (ii) If $\varphi \in \mathcal{I}$ is convex, then $\mathcal{M}_{\varphi}(\mu_i)\rightarrow\mathcal{M}_{\varphi}(\mu)$.
\end{theorem}

\begin{proof}
Let $M\in\mathcal{M}_{\varphi}(\mu)$ and $M_i\in \mathcal{M}_{\varphi}(\mu_i)$ be convex bodies such that $|{M^\circ}|=|{M^\circ_i}|=\omega_{n}$ for any $i\geq 1$,  $$\widehat{\mathcal{G}}_{\varphi}(\mu)=\int_{S^{n-1}}\varphi\big( h_{M} \big) \,d \mu\ \ \text{and}\ \ \widehat{\mathcal{G}}_{\varphi}(\mu_i)=\int_{S^{n-1}}\varphi\big( h_{M_i}\big) \,d \mu_i .$$
The weak convergence of $\mu_i\rightarrow \mu$ yields
 \begin{eqnarray} \label{continuous-1"} \nonumber\widehat{\mathcal{G}}_{\varphi}(\mu)&=&\int_{S^{n-1}}\varphi\big( h_M \big) \,d \mu\\ \nonumber&=&\lim_{i\rightarrow \infty}\int_{S^{n-1}}\varphi\big( h_M \big) \,d \mu_i\\  \nonumber
 &=&\limsup_{i\rightarrow \infty}\int_{S^{n-1}}\varphi\big( h_M \big) \,d \mu_i\\
 &\geq& \limsup_{i\rightarrow \infty}\widehat{\mathcal{G}}_{\varphi}(\mu_i).
 \end{eqnarray}
Let $R_i,$ $u_i$ and $v$ be as in the proof of Theorem \ref{Theorem3.1}, i.e.,
$R_i=\rho_{M_i}(u_i)=\max_{u\in S^{n-1}}\{\rho_{M_i}(u)\},$  $u_i\rightarrow v$ as $i\rightarrow \infty$ and $h_{M_i}(u) \geq R_i \cdot \langle u,u_i \rangle_+$ for any $u \in S^{n-1}$ and any $i\geq 1$. 

Assume $\sup_{i\geq 1} R_i=\infty,$ and without loss of generality, let $R_i\rightarrow \infty.$ Since $\mu$ is not concentrated on any hemisphere of $S^{n-1}$, there exists an integer $j_0$ such that (\ref{3.2'}) holds. By the weak convergence of $\mu_i\rightarrow \mu,$ (\ref{continuous-1"}) and Lemma \ref{uniformly converge-lemma}, one gets, for any positive constant $C>0$,
\begin{eqnarray} \label{3.6}
\nonumber \widehat{\mathcal{G}}_{\varphi}(\mu)&=& \int_{S^{n-1}}\varphi\big( h_{M} \big) \,d \mu\\ \nonumber \\ &=&\lim_{i\rightarrow \infty} \int_{S^{n-1}}\varphi\big( h_{M} \big) \,d \mu_i\nonumber \\ &\geq& \liminf_{i\rightarrow\infty}\int_{S^{n-1}}\varphi\big( h_{M_i} \big) \,d \mu_i \nonumber\\ \nonumber  
&\geq&  \liminf_{i\rightarrow \infty}\int_{S^{n-1}}\varphi\big( R_i \cdot \langle u,u_i \rangle_+ \big) \,d \mu_i(u)\\ \nonumber
&\geq& \liminf_{i\rightarrow \infty}\int_{S^{n-1}}\varphi\big( C\cdot \langle u,u_i \rangle_+ \big) \,d \mu_i(u)\\ \nonumber
&=& \int_{S^{n-1}}\varphi\big( C\cdot \langle u,v \rangle_+ \big) \,d \mu(u)\\
&\geq& \varphi\left( \displaystyle{\frac{C}{j_0}} \right)\,\cdot\int_{\{u\in S^{n-1}:\,\langle u,v \rangle_+\geq\frac{1}{j_0}\}}\langle u,v \rangle_+ \,d\mu(u). \nonumber
\end{eqnarray}
This yields a contradiction $\widehat{\mathcal{G}}_{\varphi}(\mu)\geq \infty$ if let $C\rightarrow \infty.$ Hence $\sup_{i\geq 1} R_i<\infty$ and $\{M_{i}\}_{i=1}^\infty $ is bounded.

Let $\{M_{i_k}\}_{k=1}^\infty$ be any subsequence of $\{M_i\}_{i=1}^\infty.$ By the boundedness of $\{M_{i_k}\}_{k=1}^\infty $ and Lemma \ref{blaschke-selection}, one can find a subsequence $\{M_{i_{k_j}}\}_{j=1}^\infty$ of $ \{M_{i_k}\}_{k=1}^\infty$ and a convex body $M^\prime\in \mathcal{K}_0$ such that $M_{i_{k_j}}\rightarrow M^\prime$ as $j\rightarrow \infty$ and $|{(M^\prime)}^\circ|=\omega_{n}.$ Moreover, $\varphi(h_{M_{i_{k_j}}})\rightarrow \varphi(h_{M^\prime})$ uniformly on $S^{n-1}$.

\vskip 2mm (i) Let $\{\mu_{i_k}\}_{k=1}^\infty \subseteq\{\mu_i\}_{i=1}^\infty$ be a subsequence such that $$\lim_{k \rightarrow \infty}\widehat{\mathcal{G}}_{\varphi}(\mu_{i_k}) = \liminf_{i\rightarrow \infty}\widehat{\mathcal{G}}_{\varphi}(\mu_i).$$
By the arguments above, there exist a subsequence $\{M_{i_{k_j}}\}_{j=1}^\infty$ of $\{M_{i_k}\}_{k=1}^\infty$ and a convex body $M^\prime\in \mathcal{K}_0$ such that $M_{i_{k_j}}\rightarrow M^\prime$ as $j\rightarrow \infty$ and $|{(M^\prime)}^\circ|=\omega_{n}$. Thus, by Lemma \ref{uniformly converge-lemma}, one has
\begin{eqnarray*}
\liminf_{i\rightarrow \infty}\widehat{\mathcal{G}}_{\varphi}(\mu_i)&=&\lim_{j \rightarrow \infty}\widehat{\mathcal{G}}_{\varphi}(\mu_{i_{k_j}})\\&=&\lim_{j \rightarrow \infty}\int_{S^{n-1}}\varphi\big( h_{M_{i_{k_j}}} \big) \,d \mu_{i_{k_j}}\\&=&\int_{S^{n-1}}\varphi\big( h_{M^\prime} \big) \,d \mu \\&\geq & \widehat{\mathcal{G}}_{\varphi}(\mu).
\end{eqnarray*}
Together with (\ref{continuous-1"}), one has $\widehat{\mathcal{G}}_{\varphi}(\mu_i) \rightarrow \widehat{\mathcal{G}}_{\varphi}(\mu)$ as $i\rightarrow \infty.$
 
\vskip 2mm  (ii) Let $\{M_{i_k}\}_{k=1}^\infty$ be any subsequence of $\{M_i\}_{i=1}^\infty.$ The weak convergence of $\mu_{i_k} \rightarrow \mu$, along with part (i) above, implies
\be\nonumber
\widehat{\mathcal{G}}_{\varphi}(\mu)=\lim_{k\rightarrow\infty}\widehat{\mathcal{G}}_{\varphi}(\mu_{i_k})
=\lim_{k\rightarrow\infty}\int_{S^{n-1}}\varphi\big( h_{M_{i_k}} \big) \,d \mu_{i_{k}}.
\ee 
Again, for $\{M_{i_k}\}_{k=1}^\infty$, there exist a subsequence $\{M_{i_{k_j}}\}_{j=1}^\infty $ and a convex body $M^\prime\in \mathcal{K}_0$ such that $M_{i_{k_j}}\rightarrow M^\prime$ as $j\rightarrow \infty$ and $|{(M^\prime)}^\circ|=\omega_{n}$.
By Lemma \ref{uniformly converge-lemma}, one has
$$
\widehat{\mathcal{G}}_{\varphi}(\mu)
=\lim_{j\rightarrow\infty}\widehat{\mathcal{G}}_{\varphi}(\mu_{i_{k_j}})=\lim_{j\rightarrow\infty}\int_{S^{n-1}}\varphi\big( h_{M_{i_{k_j}}} \big) \,d \mu_{i_{k_j}}
=\int_{S^{n-1}}\varphi\big( h_{M^\prime} \big) \,d \mu.
$$
The uniqueness in Theorem \ref{Theorem3.1} yields $M=M^\prime$. Consequently, $M_{i_{k_j}}\rightarrow M$ as $j \rightarrow \infty.$ In summary, we have proved that any subsequence of $\{M_{i}\}_{i=1}^\infty$ has a subsequence converging to $M,$ and then $M_i\rightarrow M$ as $i \rightarrow \infty$ with respect to the Hausdorff metric.
\end{proof}

Let $\mathcal{D}$ be the set of continuous functions $\varphi: (0,\infty)\rightarrow(0,\infty)$ such that
$\varphi$ is strictly decreasing, $\lim_{t \to 0^+}\varphi(t)=\infty, \varphi(1)=1$ and $\lim_{t \to \infty}\varphi(t)=0.$ The following proposition states that the polar Orlicz-Minkowski problems (\ref{Orlicz-Minkowski problem})
might not be solvable in general for cases other than (\ref{polar Orlicz-Minkowski problem 3.2}). For an $n\times n$ matrix $T$, its transpose is denoted by $T^t$. Denote by $O(n)$ the set of all invertible $n\times n$ matrices such that $T T^t=T^tT$ equal to identity matrix on $\mathbb{R}^n$. 
\begin{proposition}\label{denial of other possibilities}
Let $\mu=\sum^m_{i=1}\lambda_i\delta_{u_i}$ with $\lambda_i>0$ for any $1\leq i\leq m$ be a given nonzero finite discrete measure whose support $\{u_1,u_2,\cdots, u_m\}$ is not concentrated on any hemisphere of $S^{n-1}$.\\
(i) If $\varphi\in\mathcal{D}$ and the first coordinates of $u_1,u_2,\cdots, u_m$ are all nonzero, then
$$\inf \bigg\{\int_{S^{n-1}}\varphi\big( h_L \big) \,d \mu: L \in \mathcal{K}_{0} \ \text{and}\   |L^\circ|=\omega_{n}\bigg\}=0.$$
(ii) If $\varphi\in\mathcal{I}\cup\mathcal{D},$ then
 $$\sup \bigg\{\int_{S^{n-1}}\varphi\big( h_L\big) \,d \mu: L \in \mathcal{K}_{0} \ \text{and}\   |L^\circ|=\omega_{n}\bigg\}=\infty.$$
\end{proposition}
\begin{proof}
(i) Let $\alpha=\min_{1\leq i\leq m}\{|(u_i)_1|\}$ and then $\alpha>0$ by assumption. For any $\epsilon>0$, let
$$T_\epsilon=\text{diag}(1,1,\cdots,1,\epsilon^{n}) \ \ \text{and} \ \ L_\epsilon=\epsilon^{-1} \cdot T_\epsilon B^n_2.$$
Thus $(L_\epsilon)^\circ=\epsilon \cdot (T^{t}_\epsilon)^{-1} B^n_2$ and $|(L_\epsilon)^\circ|=\omega_n$. Moreover, for $1\leq i\leq m$, one has
$$|T_\epsilon u_i|=\sqrt{(u_i)^2_1+(u_i)^2_2+\cdots+\epsilon^{2n}(u_i)^2_n}\geq |(u_i)_1|\geq\alpha,
$$
and then
$$
h_{L_\epsilon}(u_i)=\max_{v_1\in L_\epsilon}\langle v_1,u_i \rangle=\max_{v_2\in B^n_2}\langle \epsilon^{-1}\cdot T_\epsilon v_2,u_i \rangle=\epsilon^{-1}\cdot \max_{v_2\in B^n_2}\langle v_2, T_\epsilon u_i \rangle=\epsilon^{-1}\cdot |T_\epsilon u_i|\geq \alpha/\epsilon.$$
It follows from the fact $\varphi$ is strictly decreasing that
\begin{eqnarray*}
\int_{S^{n-1}}\varphi\big( h_{L_\epsilon}\big)\,d\mu
&=&\sum^m_{i=1}\varphi\left(h_{L_\epsilon}(u_i) \right)\cdot\mu(\{u_i\})\\
&\leq&\sum^m_{i=1}\varphi\left(\alpha/\epsilon\right)\cdot\mu(\{u_i\})\\&=&\varphi\left(\alpha/\epsilon \right)\cdot\mu(S^{n-1}).
\end{eqnarray*}
Note that $\varphi(\alpha/\epsilon)\rightarrow 0$ as $\epsilon\rightarrow 0$ due to $\varphi\in \mathcal{D}$, and then
$$
\inf \bigg\{\int_{S^{n-1}}\varphi\big( h_L \big) \,d \mu: L \in \mathcal{K}_{0} \ \text{and}\   |L^\circ|=\omega_{n}\bigg\}
\leq\varphi\left(\alpha/\epsilon \right)\cdot\mu(S^{n-1})\rightarrow0\ \ \text{as} \ \epsilon\rightarrow 0.
$$

(ii) Without loss of generality, we assume $\mu(\{u_1\})>0.$ By the Gram-Schmidt process, one could get an orthogonal matrix $T\in O(n)$ with its first column vector being $u_1.$ For any $\epsilon>0,$ let
$$T_\epsilon=T\cdot\text{diag}(\epsilon^{-1},\epsilon,1,1,\cdots,1)\cdot T^t \ \ \text{and} \ \ L_\epsilon=T_\epsilon B^n_2.$$
Then  $|(L_\epsilon)^\circ|=\omega_n$ and
$$
h_{L_\epsilon}(u_1)=\max_{v_1\in L_\epsilon}\langle v_1,u_1 \rangle=\max_{v_2\in B^n_2}\langle T_\epsilon v_2,u_1 \rangle=\max_{v_2\in B^n_2}\langle v_2, T_\epsilon u_1 \rangle=\max_{v_2\in B^n_2}\langle v_2,\epsilon^{-1} u_1 \rangle=1/\epsilon.$$
Thus, one has
\begin{eqnarray*}
\int_{S^{n-1}}\varphi\big( h_{L_\epsilon} \big)\,d\mu
=\sum^m_{i=1}\varphi\left(h_{L_\epsilon}(u_i) \right)\cdot\mu(\{u_i\})
\geq\varphi\left( h_{L_\epsilon}(u_1)\right)\cdot \mu(\{u_1\})
=\varphi\left(1/ \epsilon \right)\cdot \mu(\{u_1\}).
\end{eqnarray*} For $\varphi\in\mathcal{I}$, letting $\epsilon\rightarrow 0$, one gets
$$\sup \bigg\{\int_{S^{n-1}}\varphi\big( h_L\big) d \mu: L \in \mathcal{K}_{0} \ \text{and}\   |L^\circ|=\omega_{n}\bigg\}=\infty.$$
Following the same lines, the desired result for $\varphi\in\mathcal{D}$ can be obtained as well.
\end{proof}
Let $\phi:(0,\infty)\rightarrow(0,\infty)$ be a continuous function and $K\in\mathcal{K}_0,$ then $\,d\mu=\displaystyle{\frac{1}{\phi(h_K)}}\,dS(K,\cdot)\in\Omega$ is a nonzero finite measure on $S^{n-1}$ which is not concentrated on any hemisphere of $S^{n-1}$. Theorem \ref{Theorem3.1} yields that if $\varphi\in\mathcal{I},$ there exists a convex body $M\in\mathcal{K}_0$ such that $|M^\circ|=\omega_n$ and
$$nV_{\varphi,\phi}(K,M)=\int_{S^{n-1}}\frac{\varphi(h_M(u))}{\phi(h_K(u))}\,dS(K,u)=\inf\bigg\{nV_{\varphi,\phi}(K,L):L\in\mathcal{K}_0\ \ \text{and}\ \ |L^\circ|=\omega_n\bigg\}.$$
This is the polar analogue of the Orlicz-Minkowski problems studied in \cite{HLYZ2010,HH2012,Li2014}.
In particular, if $\varphi(t)=t^q$ and $\phi(t)=t^{q-1}$ for $q>0,$ it goes back to (\ref{1.7}). Other examples include the measures generated from $S_\mathcal{V}(K, \cdot)$ for some $K\in \mathcal{K}_0$ as given by Definition \ref{definition 2.1}.   For instance, let $d\mu=\displaystyle{\frac{1}{\phi(h_K)}}d\mu_\tau(K,\cdot),$ which is not concentrated on any hemisphere of $S^{n-1}.$ Theorem \ref{Theorem3.1} implies the existence of a convex body $M\in\mathcal{K}_0$ such that $|M^\circ|=\omega_n$ and
$$\int_{S^{n-1}}\frac{\varphi(h_M(u))}{\phi(h_K(u))}\,d\mu_\tau(K,u)=\inf\bigg\{\int_{S^{n-1}}\frac{\varphi(h_L(u))}{\phi(h_K(u))}\,d\mu_\tau(K,u):L\in\mathcal{K}_0\ \ \text{and}\ \ |L^\circ|=\omega_n\bigg\}.$$
In particular, if $\varphi(t)=t^q$ and $\phi(t)=t^{q-1}$ for $q>0,$ similar to (\ref{1.7}), one gets the $L_q$ torsional Petty body; namely, $M\in\mathcal{K}_0$ such that $|M^\circ|=\omega_n$ and
\begin{eqnarray*}\mu_{\tau,q}(K,M)&=&\int_{S^{n-1}}\Big(\frac{h_M(u)}{h_K(u)}\Big)^q h_K(u)d\mu_\tau(K,u)\\&=&\inf\bigg\{\int_{S^{n-1}}\Big(\frac{h_L(u)}{h_K(u)}\Big)^q h_K(u)d\mu_\tau(K,u):L\in\mathcal{K}_0\ \text{and}\ \ |L^\circ|=\omega_n\bigg\}.\end{eqnarray*}
Of course, one can also let $d\mu=\displaystyle{\frac{1}{\phi(h_K)}}d\mu_p(K,\cdot)$ and get the similar results for the $p$-capacitary measure. Different but closely related concepts, the $p$-capacitary Orlicz-Petty bodies,  will be discussed in Section \ref{section:4}.

It is well known that $\mu\in\Omega,$ i.e., $\mu$ is not concentrated on any hemisphere of $S^{n-1}$, is the minimal requirement for solutions to various Minkowski problems. For instance, for $q>1,$ it has been proved in \cite{HLYZ2005} that if $\mu\in\Omega,$ there exists a convex body $K$ containing $o$ (note that $K$ may not be in $\mathcal{K}_0$ unless $q>n$) such that $|K|\cdot h^{q-1}_Kd\mu=dS(K,\cdot).$
In this case, especially if $K\in\mathcal{K}_0,$ one can link $\mu$ to a convex body. However, it is not clear whether, in general, there exists a convex body $K\in\mathcal{K}_0$ such that $d\mu=c\cdot h^{1-q}_KdS(K,\cdot)$ for $q<1$; see special cases in \cite{HLYZ2010, GZhu2015I,GZhu2015II,GZhupreprint}. In other words, $\mu\in\Omega,$ although closely related to convex bodies, is in fact more general than the measures generated from convex bodies. Consequently, the polar Orlicz-Minkowski problem is much more general than (\ref{1.7}) and its (direct Orlicz) extensions involving convex bodies.

Now let us discuss some dissimilarities between the Minkowski and the polar Minkowski problems. First of all, the solutions are always convex bodies in $\mathcal{K}_0$ for the polar Minkowski problem (\ref{polar Orlicz-Minkowski problem 3.2}), while this may not be true for Minkowski problems as mentioned above. Secondly, as showed in Proposition \ref{denial of other possibilities}, the solutions to the polar Minkowski problems for discrete measures usually do not exist, except the case (\ref{polar Orlicz-Minkowski problem 3.2}) for $\varphi\in \mathcal{I}$. However, as showed in, e.g. \cite{GZhu2015II,GZhupreprint}, the solutions to the $L_q$ Minkowski problems for discrete measures could be well-existed for $q<0.$ Finally, it seems intractable to find a direct relation between $\mu\in\Omega$ and the solutions to the polar Minkowski problems, while such a relation usually can be established as long as the solutions exist for the related Minkowski problems.

Define $\|f\|_{L_\varphi(\mu)}$ as follows:
for $\mu\in\Omega$ and $f: S^{n-1}\rightarrow \mathbb{R},$
\begin{eqnarray*}  \|f\|_{L_\varphi(\mu)} &=& \inf\bigg\{\lambda>0:\int_{S^{n-1}}\varphi\Big(\frac{f}{\lambda}\Big)\,d\mu\leq\mu(S^{n-1})\bigg\}\ \ \text{for} \ \ \varphi\in\mathcal{I}, \\ \|f\|_{L_\varphi(\mu)}&=&\inf\bigg\{\lambda>0:\int_{S^{n-1}}\varphi\Big(\frac{f}{\lambda}\Big)\,d\mu\geq\mu(S^{n-1})\bigg\}\ \ \text{for} \ \ \varphi\in\mathcal{D}. \end{eqnarray*}  Clearly, $\|t \cdot f\|_{L_\varphi(\mu)}=t\cdot \|f\|_{L_\varphi(\mu)}$ for any $t>0.$ Moreover, for $L\in\mathcal{K}_0$ and $\varphi\in \mathcal{I}\cup\mathcal{D}$,  $$\|h_L\|_{L_\varphi(\mu)}>0\ \ \ \mathrm{and} \ \ \  \int_{S^{n-1}}\varphi\Big(\frac{h_L}{\|h_L\|_{L_\varphi(\mu)}}\Big)\,d\mu=\mu(S^{n-1}).$$
Let ${\left\{L_i\right\}}_{i=1}^{\infty}\subseteq \mathcal{K}_0$ and $L\in\mathcal{K}_0$  be such that $L_i\rightarrow L$ with respect to the Hausdorff metric, then $\|h_{L_i}\|_{L_\varphi(\mu)}\rightarrow \|h_L\|_{L_\varphi(\mu)}.$ This can be proved along the same lines as Proposition \ref{Proposition3.1} in Section \ref{section:4}. Moreover, if $\varphi\in\mathcal{I},$ $\mu_i\rightarrow\mu$ weakly and there exists a positive constant $C>0,$ such that,  $\|h_{L_i}\|_{L_\varphi(\mu_i)}\leq C$ for any $i\geq1,$ then ${\left\{L_i\right\}}_{i=1}^{\infty}$ is uniformly bounded; this can be proved along the same lines as Proposition \ref{proposition3.3}. We leave the details for readers. 

For $\|h_L\|_{L_\varphi(\mu)}$, we can also ask the related polar Orlicz-Minkowski problems:
under what conditions on $\varphi$ and $\mu$, there exists a convex body $M\in\mathcal{K}_0$ such that $M$ is an optimizer of
\begin{eqnarray}
&&\inf \big\{\|h_L\|_{L_\varphi(\mu)}: L \in \mathcal{K}_{0} \ \text{and}\   |L^\circ|=\omega_{n}\big\},\label{3.8} \ \ \ \text{or} \\
&&\sup \big\{\|h_L\|_{L_\varphi(\mu)}: L \in \mathcal{K}_{0} \ \text{and}\   |L^\circ|=\omega_{n}\big\}.\label{3.8'}\end{eqnarray}
The following theorem states that problem (\ref{3.8}) is solvable for $\varphi\in\mathcal{I}$ and $\mu\in \Omega$. The proof follows along the same lines as Theorem \ref{Theorem3.1} and will be omitted.
\begin{theorem}\label{Theorem3.3}
Let $\mu\in\Omega$  and $\varphi \in \mathcal{I}.$ Then there exists a convex body $\widehat{M}\in\mathcal{K}_0$ such that $|\widehat{M}^\circ|=\omega_{n}$ and
$$\|h_{\widehat{M}}\|_{L_\varphi(\mu)}=\inf\big\{\|h_L\|_{L_\varphi(\mu)}: L \in \mathcal{K}_{0} \ \text{and}\   |L^\circ|=\omega_{n}\big\}.$$
Moreover, if $\varphi\in\mathcal{I}$ is convex, then $\widehat{M}$ is the unique solution to the polar Orlicz-Minkowski problem $(\ref{3.8})$.
\end{theorem}
Moreover, one can get arguments similar to Theorem \ref{Theorem 3.2}.
When $\mu$ is a discrete measure, part (i) of the following proposition states that the solutions to problem (\ref{3.8}) for $\varphi\in\mathcal{I}$ are polytopes. However, parts (ii)-(iv) show that the polar Orlicz-Minkowski problems (\ref{3.8}) and (\ref{3.8'}) might not be solvable in general.

\begin{proposition}Let $\mu=\sum^m_{i=1}\lambda_i\delta_{u_i}$ with $\lambda_i>0$ for any $1\leq i\leq m$ be a given nonzero finite discrete measure whose support $\{u_1,u_2,\cdots, u_m\}$ is not concentrated on any hemisphere of $S^{n-1}$.\\
(i) If $\varphi\in\mathcal{I}$ and $\widehat{M}\in\mathcal{K}_0$ is a solution to the polar Orlicz-Minkowski problem $(\ref{3.8})$, then $\widehat{M}$ is a polytope with $u_1,u_2,\cdots, u_m$ being the unit normal vectors of its faces.\\
(ii) If $\varphi\in\mathcal{I},$ then
$$\sup\big\{\|h_L\|_{L_\varphi(\mu)}: L \in \mathcal{K}_{0} \ \text{and}\   |L^\circ|=\omega_{n}\big\}=\infty.$$
 (iii) If $\varphi\in\mathcal{D}$, then
$$\inf\big\{\|h_L\|_{L_\varphi(\mu)}: L \in \mathcal{K}_{0} \ \text{and}\   |L^\circ|=\omega_{n}\big\}=0.$$
 (iv) If $\varphi\in\mathcal{D}$ and the first coordinates of $u_1,u_2,\cdots, u_m$ are all nonzero, then
$$\sup\big\{\|h_L\|_{L_\varphi(\mu)}: L \in \mathcal{K}_{0} \ \text{and}\   |L^\circ|=\omega_{n}\big\}=\infty.$$
\end{proposition}

\section{The $p$-capacitary Orlicz-Petty bodies}\label{section:4}
Theorem \ref{Theorem3.1}, by letting  $\varphi(t)=t^q$ and $\,d\mu=\displaystyle{\frac{p-1}{n-p}}h^{1-q}_K \,d\mu_p(K,\cdot)$ for $q>0$ and $p\in (1, n)$, implies the existence of a convex body $M\in\mathcal{K}_0,$ which will be called the  $p$-capacitary $L_q$ Petty body, such that $|M^\circ|=\omega_n$ and
\begin{eqnarray*}C_{p,q}(K,M) = \frac{p-1}{n-p}\int_{S^{n-1}}\Big(\frac{h_M(u)}{h_K(u)}\Big)^q h_K(u)\,d\mu_p(K,u) = \inf\big\{C_{p, q}(K, L): L\in\mathcal{K}_0\ \text{and}\ |L^\circ|=\omega_n\big\}.\end{eqnarray*}
This motivates our interest in studying the $p$-capacitary Orlicz-Petty bodies.
\subsection{The nonhomogeneous and homogeneous Orlicz mixed $p$-capacities}
For $\varphi\in\mathcal{I}\cup\mathcal{D},$ the nonhomogeneous $L_\varphi$ Orlicz mixed $p$-capacity $C_{p,\varphi}(\cdot,\cdot)$ defined in (\ref{1.18}) was introduced in \cite{HYZ2017I}. When $\varphi(t)=t,$ the mixed $p$-capacity was provided in \cite{Cole2015}.
\begin{definition}\label{nonhomogeneous Orlicz mixed $p$-capacity}
Let $\varphi\in\mathcal{I}\cup\mathcal{D},$ $p\in(1,n)$, $K\in \mathcal{K}_0$ and $L\in \mathcal{S}_0$, define  $C_{p,\varphi}(K,L^\circ)$ by
\begin{equation}\label{definition of mixed volume for star body nonhomogeneous}
C_{p,\varphi}(K,L^\circ)=\frac{p-1}{n-p}\int_{S^{n-1}}\varphi\left( \frac{1}{\rho_L(u)\cdot h_K(u)} \right)h_K(u)\,d \mu_p(K,u).
\end{equation}\end{definition} When $L\in \mathcal{K}_0$, by the bipolar theorem, one can easily get  (\ref{1.18}) from (\ref{definition of mixed volume for star body nonhomogeneous}) by  $C_{p,\varphi}(K,L)=C_{p,\varphi}(K, (L^\circ)^\circ)$. Note that $\varphi$ in Definition \ref{nonhomogeneous Orlicz mixed $p$-capacity} can be any continuous function. However, the monotonicity of $\varphi$ is crucial in later context, so we only focus on $\varphi\in\mathcal{I}\cup\mathcal{D}.$ We would like to mention that Hong, Ye and Zhang in \cite{HYZ2017I} provided a geometric interpretation of the Orlicz mixed $p$-capacity of $K, L\in \mathcal{K}_0$. Clearly, $C_{p,\varphi}(K,K)=C_p(K)$ for $\varphi\in\mathcal{I}\cup\mathcal{D}.$ Moreover, for any $r>0$,
$$
C_{p,\varphi}(rB_2^n,B_2^n)=r^{n-p}\cdot \varphi \left(\frac{1}{r} \right)\cdot C_p (B_2^n)\ \ \text{and} \ \  C_{p,\varphi}(B_2^n,rB_2^n)=\varphi \left(r\right)\cdot C_p(B_2^n).
$$
These imply that $C_{p,\varphi}(\cdotp,\cdotp)$ is nonhomogeneous on $K$ and $L,$ if $\varphi$ is not a homogeneous function. The homogeneous analogue \cite{HYZ2017I}  is defined as follows.

\begin{definition} \label{definition-homogeneous}
Let $\varphi \in \mathcal{I}\cup \mathcal{D}$, $p\in(1,n)$, $K\in \mathcal{K}_0$ and $L\in \mathcal{S}_0$. Define $\widehat{C}_{p,\varphi}(K,L^\circ)$ by
\begin{equation}\label{definition of mixed volume for star body homogeneous}
\int_{S^{n-1}}\varphi\left( \frac{C_p(K)}{\widehat{C}_{p,\varphi}(K,L^\circ)\cdot\rho_L(u) \cdot h_{K}(u)} \right)d \mu^*_p(K,u)=1,
\end{equation}
where $\mu^*_p(K,\cdot)$ is the probability measure on $S^{n-1}$ associated with $K\in\mathcal{K}_0$ given in (\ref{measure-37}).
 \end{definition} Again, if $L\in \mathcal{K}_0$, one recovers the one given by (\ref{1.19}).   In fact, it can be easily checked  that, for $K, L\in \mathcal{K}_0$,  the following function 
$$G(\eta)=\int_{S^{n-1}}\varphi\left(\frac{C_p(K)\cdot h_{L}(u)}{\eta\cdot h_{K}(u)} \right)d \mu^*_p(K,u)$$
is continuous, strictly monotonic on $(0,\infty)$ and the range of $G(\eta)$ is $(0,\infty).$ These imply that $\widehat{C}_{p,\varphi}(K,L)$ is well-defined. Thus, for any $K,L \in \mathcal{K}_0,$ $\widehat{C}_{p,\varphi}(K,L)>0.$ In addition, as $\varphi(1)=1$ and $\mu^*_p(K,\cdot)$ is a probability measure on $S^{n-1}$, then for any $K\in\mathcal{K}_0,$ $\widehat{C}_{p,\varphi}(K,K)=C_p(K).$ Similar arguments hold for $\widehat{C}_{p,\varphi}(K,L^\circ)$ if $L\in \mathcal{S}_0$.  The following result for the homogeneity of $\widehat{C}_{p,\varphi}(\cdot,\cdot)$ follows immediately from (\ref{1.19}) and (\ref{definition of mixed volume for star body homogeneous}).

\begin{corollary} \label{corollary3.1}
Let $K,L \in \mathcal{K}_0$ and $s,t>0$. If $\varphi\in\mathcal{I} \cup \mathcal{D}$, then
$$
\widehat{C}_{p,\varphi}(s K,tL)=s^{n-p-1}\cdot t\cdot\widehat{C}_{p,\varphi}(K,L).
$$
When $L\in\mathcal{S}_0,$ then
$$
\widehat{C}_{p,\varphi}(s K,(tL)^\circ)=s^{n-p-1}\cdot t^{-1}\cdot\widehat{C}_{p,\varphi}(K,L^\circ).
$$
\end{corollary}

The following proposition deals with the continuity of $C_{p,\varphi}(\cdot,\cdot)$ and $\widehat{C}_{p,\varphi}(\cdot,\cdot)$.

\begin{proposition}\label{Proposition3.1}
Let ${\left\{K_i\right\}}_{i=1}^{\infty}\subseteq \mathcal{K}_0$ and ${\left\{L_i\right\}}_{i=1}^{\infty}\subseteq \mathcal{K}_0$ be two sequences of convex bodies such that $K_i\rightarrow K\in \mathcal{K}_0$ and $L_i\rightarrow L\in \mathcal{K}_0$ as $i\rightarrow\infty$. If $\varphi\in \mathcal{I}\cup\mathcal{D}$, then\\
$$C_{p,\varphi}(K_i,L_i)\rightarrow C_{p,\varphi}(K,L)\ \ \mbox{and}\ \ \ \widehat{C}_{p,\varphi}(K_i,L_i)\rightarrow \widehat{C}_{p,\varphi}(K,L)\ \ \ \text{as}\ \ i\rightarrow\infty.$$
\end{proposition}
\begin{proof}
As $L_i\rightarrow L$, then $h_{L_i}\rightarrow h_L$ uniformly on $S^{n-1}$. Similarly, the convergence of $K_i\rightarrow K$ implies that $h_{K_i}\rightarrow h_K$ uniformly on $S^{n-1},$ $C_p(K_i)\rightarrow C_p(K)$ and $\mu_p(K_i,\cdotp)$ converges weakly to $\mu_p(K,\cdotp)$ (see e.g. \cite{Cole2015}). In addition, there exist two constants $r,R>0,$ such that, for any $i\geq1$
\be\label{bounded-uniformly1}r\cdot B^n_2 \subseteq K_i,K,L_i,L\subseteq R\cdot B^n_2,\ee
and hence for any $i\geq 1$ and $u\in S^{n-1},$
\be\label{bounded-uniformly2}
\frac{r}{R}\leq \frac{h_{L_i}(u)}{h_{K_i}(u)}, \  \frac{h_{L}(u)}{h_{K}(u)}\leq \frac{R}{r}.
\ee
Since $\varphi$ is continuous on the interval $\bigg[\displaystyle{\frac{r}{R},\frac{R}{r}}\bigg]$, then
$$\varphi\bigg(\frac{h_{L_i}(u)}{h_{K_i}(u)}\bigg)\rightarrow \varphi\bigg(\frac{h_{L}(u)}{h_{K}(u)}\bigg)\ \ \ \text{uniformly on} \ \ S^{n-1}.$$
Together with Lemma \ref{uniformly converge-lemma}, one gets
$$ \frac{p-1}{n-p}\int_{S^{n-1}}\varphi\bigg(\frac{h_{L_i}(u)}{h_{K_i}(u)}\bigg)h_{K_i}(u)\,d \mu_p(K_i,u)\rightarrow \frac{p-1}{n-p}\int_{S^{n-1}}\varphi\bigg(\frac{h_{L}(u)}{h_{K}(u)}\bigg)\,h_K(u)\,d \mu_p(K,u),$$
and hence $C_{p,\varphi}(K_i,L_i)\rightarrow C_{p,\varphi}(K,L)$ as $i\rightarrow\infty$.

For the case $\widehat{C}_{p,\varphi}(\cdot,\cdot),$ we only prove the case $\varphi\in\mathcal{I},$
and the case $\varphi\in\mathcal{D}$ follows along the same argument. 
It follows from the monotonicity of $C_p(\cdot)$ and $\varphi$, (\ref{bounded-uniformly1}) and (\ref{bounded-uniformly2}) that
\begin{eqnarray*} 
1&=&\int_{S^{n-1}}\varphi\bigg(\frac{C_p(K_i)\cdot h_{L_i}(u)}{\widehat{C}_{p,\varphi}(K_i,L_i)\cdot h_{K_i}(u)} \bigg)\,d \mu^*_p(K_i,u)\leq \varphi\bigg(\frac{C_p(R\cdot B_2^n)\cdot R}{\widehat{C}_{p,\varphi}(K_i,L_i)\cdot r}\bigg);\\
1&=&\int_{S^{n-1}}\varphi\bigg( \frac{C_p(K_i)\cdot h_{L_i}(u)}{\widehat{C}_{p,\varphi}(K_i,L_i)\cdot h_{K_i}(u)} \bigg)\,d \mu^*_p(K_i,u)\geq\varphi\bigg(\frac{C_p(r\cdot B_2^n)\cdot r}{\widehat{C}_{p,\varphi}(K_i,L_i)\cdot R}\bigg).
\end{eqnarray*} 
Combining with the fact that $\varphi(1)=1$, one gets, for any $ i\geq1$,
$$
0< \frac{C_p(r\cdot B_2^n)\cdot r}{R} \leq\widehat{C}_{p,\varphi}(K_i,L_i)\leq \frac{C_p(R\cdot B_2^n)\cdot R}{r}<\infty.
$$

Let
$$S=\limsup_{i\rightarrow \infty}\widehat{C}_{p,\varphi}(K_i,L_i)<\infty\ \ \text{and}\ \ I=\liminf_{i\rightarrow \infty}\widehat{C}_{p,\varphi}(K_i,L_i)>0.$$
Thus there exists a subsequence $\left\{\widehat{C}_{p,\varphi}(K_{i_k},L_{i_k})\right\}_{k=1}^{\infty}$ such that
$$
\frac{k}{k+1}S<\widehat{C}_{p,\varphi}(K_{i_k},L_{i_k})\,\,\ \text{for any} \ k \geq1\,\, \text{and} \, \lim_{k \rightarrow \infty}\widehat{C}_{p,\varphi}(K_{i_k},L_{i_k})=S.$$
These along with the fact that $\varphi$ is increasing and Lemma \ref{uniformly converge-lemma} yield
\begin{eqnarray*}\nonumber
1&=&\lim_{k\rightarrow \infty}\int_{S^{n-1}}\varphi \bigg(\frac{C_p(K_{i_k})\cdot h_{L_{i_k}}(u)}{\widehat{C}_{p,\varphi}(K_{i_k},L_{i_k})\cdot h_{K_{i_k}}(u)} \bigg)\,d \mu^*_p(K_{i_k},u)\\ \nonumber
&\leq& \lim_{k\rightarrow \infty}\int_{S^{n-1}}\varphi \bigg(\frac{(k+1)\cdot C_p(K_{i_k})\cdot h_{L_{i_k}}(u)}{k\cdot S\cdot h_{K_{i_k}}(u)} \bigg)\,d \mu^*_p(K_{i_k},u)\\ \nonumber
&=&\int_{S^{n-1}}\varphi \bigg(\frac{C_p(K)\cdot h_{L}(u)}{S\cdot h_{K}(u)}\bigg)\,d \mu^*_p(K,u).
\end{eqnarray*}
Similarly, there exists a subsequence $\left\{\widehat{C}_{p,\varphi}(K_{i_l},L_{i_l})\right\}_{l=1}^{\infty}$ such that
$$
\frac{l+1}{l}I>\widehat{C}_{p,\varphi}(K_{i_l},L_{i_l})\,\,\ \,\text{for any} \  l \geq1\,\  \text{and} \, \ \lim_{l \rightarrow \infty}\widehat{C}_{p,\varphi}(K_{i_l},L_{i_l})=I.$$
Hence
\begin{eqnarray*}\nonumber
1&=&\lim_{l\rightarrow \infty}\int_{S^{n-1}}\varphi \bigg(\frac{C_p(K_{i_l})\cdot h_{L_{i_l}}(u)}{\widehat{C}_{p,\varphi}(K_{i_l},L_{i_l})\cdot h_{K_{i_l}}(u)}\bigg)\,d \mu^*_p(K_{i_l},u)\\ \nonumber
&\geq& \lim_{l\rightarrow \infty}\int_{S^{n-1}}\varphi \bigg(\frac{l\cdot C_p(K_{i_l})\cdot h_{L_{i_l}}(u)}{(l+1)\cdot I\cdot h_{K_{i_l}}(u)}\bigg)\,d \mu^*_p(K_{i_l},u)\\ \nonumber
&=&\int_{S^{n-1}}\varphi \bigg(\frac{C_p(K)\cdot h_{L}(u)}{I\cdot h_{K}(u)}\bigg)\,d \mu^*_p(K,u).
\end{eqnarray*}
Together with (\ref{1.19}), one gets:
$$\limsup_{i\ \rightarrow \infty}\widehat{C}_{p,\varphi}(K_i,L_i)\leq \widehat{C}_{p,\varphi}(K,L)\leq \liminf_{i\rightarrow \infty}\widehat{C}_{p,\varphi}(K_i,L_i),$$  and hence $\widehat{C}_{p,\varphi}(K_i,L_i)\rightarrow \widehat{C}_{p,\varphi}(K,L)$ as $i\rightarrow\infty$ as desired.
\end{proof}
 
The following proposition is needed.
\begin{proposition}\label{proposition3.3}
Let $\{K_i\}_{i=1}^\infty\subseteq\mathcal{K}_0$ and $K\in \mathcal{K}_0$ be such that $K_i \rightarrow K$ as $i\rightarrow \infty.$ Let $\{M_i\}_{i=1}^\infty \subseteq \mathcal{K}_0$ and $\varphi\in\mathcal{I}$ be such that $\{C_{p,\varphi}(K_i,M_i)\}_{i=1}^\infty$ or $\{\widehat{C}_{p,\varphi}(K_i,M_i)\}_{i=1}^\infty$ is bounded. Then $\{M_i\}_{i=1}^\infty$ is uniformly bounded.
\end{proposition}
\begin{proof} Note that $\mu_p(K,\cdot)$ is not concentrated on any hemisphere of $S^{n-1}$. Hence, for any given $v\in S^{n-1}$, 
\begin{equation*}  
0<\int_{S^{n-1}}\langle u,v \rangle_+\,d \mu_p(K,u)=\lim_{j\rightarrow \infty}\int_{\{u\in S^{n-1}:\,\langle u,v \rangle_+\geq\frac{1}{j}\}}\langle u,v \rangle_+ \,d \mu_p(K,u).
\end{equation*}
Thus there exists an integer $j_0 \in N$ such that
\begin{equation} \label{not concentrated on any hemisphere}
\int_{\{u\in S^{n-1}:\,\langle u,v \rangle_+\geq\frac{1}{j_0}\}}\langle u,v \rangle_+ \,d \mu_p(K,u)>0.
\end{equation}

As $K_i\rightarrow K,$ then $h_{K_i}\rightarrow h_K$ uniformly on $S^{n-1},$  $C_p(K_i)\rightarrow C_p(K)$ and $\mu_p(K_i,\cdotp)$ converges weakly to $\mu_p(K,\cdotp).$ Again, one can find $r_0, R_0>0$, such that, $r_0 \leq h_K(u), \, h_{K_i}(u)\leq R_0$  for any $i\geq 1$ and any $u\in S^{n-1}$.  Let $R_i=\rho_{M_i}(u_i)=\max_{u \in S^{n-1}} \{ \rho_{M_i}(u) \}$ and then  
$h_{M_i}(u) \geq R_i \cdot \langle u,u_i \rangle_+$  for any $u \in S^{n-1}$ and any $i\geq 1$. 
As $S^{n-1}$ is compact, without loss of generality, let $u_i \rightarrow v\in S^{n-1}$ as $i\rightarrow\infty$.

Firstly, we consider the case that $\{\widehat{C}_{p,\varphi}(K_i,M_i)\}_{i=1}^\infty$ is bounded. Then there exists a constant $B>0$ such that $B \geq \widehat{C}_{p,\varphi}(K_i,M_i)$ for any $i\geq1$. Suppose that $M_i$ is not bounded uniformly, i.e., $\sup_{i\geq 1} R_i=\infty.$  Without loss of generality, assume $R_{i} \rightarrow \infty$ as $i\rightarrow \infty$.  By the monotonicity of $\varphi$ and Lemma \ref{uniformly converge-lemma}, one has, for any constant $C>0$, 
\begin{eqnarray*}
1&=&\lim_{i\rightarrow \infty}\int_{S^{n-1}}\varphi \bigg(\frac{C_p(K_i)\cdot h_{M_{i}}(u)}{\widehat{C}_{p,\varphi}(K_{i},M_{i})\cdot h_{K_{i}}(u)}\bigg)\,d\mu^*_p(K_{i},u)\\
&\geq& \liminf_{i\rightarrow \infty}\int_{S^{n-1}}\varphi \bigg(\frac{C_p(K_i)\cdot R_i \cdot \langle u,u_i \rangle_+}{B\cdot R_0}\bigg)\,d \mu^*_p(K_{i},u)\\
&\geq& \liminf_{i\rightarrow \infty}\int_{S^{n-1}}\varphi \bigg(\frac{C_p(K_i)\cdot C \cdot \langle u,u_i \rangle_+}{B\cdot R_0}\bigg)\,d \mu^*_p(K_{i},u)\\
&=&\int_{S^{n-1}}\liminf_{i\rightarrow \infty}\varphi \bigg(\frac{C_p(K_i)\cdot C \cdot \langle u,u_i \rangle_+}{B\cdot R_0}\bigg)\,d \mu^*_p(K,u)\\
&=&\int_{S^{n-1}}\varphi \bigg(\frac{C_p(K)\cdot C\cdot \langle u,v \rangle_+}{B\cdot R_0}\bigg)\,d \mu^*_p(K,u)\\
&\geq& \varphi \bigg(\frac{C_p(K)\cdot C}{B\cdot R_0\cdot j_0}\bigg)\,\frac{(p-1)\cdot r_0}{(n-p)\cdot C_p(K)}\cdot\int_{\{u\in S^{n-1}:\,\langle u,v \rangle_+\geq\frac{1}{j_0}\}}\langle u,v \rangle_+ \,d \mu_p(K,u).
\end{eqnarray*}
A contradiction $1\geq \infty$ is obtained by  (\ref{not concentrated on any hemisphere}) and letting $C\rightarrow \infty$. Hence $\sup_{i\geq 1} R_i<\infty$.

Similarly, if $\{C_{p,\varphi}(K_i,M_i)\}_{i=1}^\infty$ is bounded, then there exists a positive constant $B>0$ such that $B \geq C_{p,\varphi}(K_i,M_i)$ for any $i\geq1$. Thus, for any given constant $C>0$, one has
\begin{eqnarray*}
B&\geq& \liminf_{i\rightarrow\infty}\frac{p-1}{n-p}\int_{S^{n-1}}\varphi \bigg(\frac{h_{M_{i}}(u)}{ h_{K_{i}}(u)}\bigg)h_{K_i}(u)\,d\mu_p(K_{i},u)\\
&\geq &\liminf_{i\rightarrow \infty}\frac{p-1}{n-p}\int_{S^{n-1}}\varphi \bigg(\frac{ C \cdot \langle u,u_i \rangle_+}{R_0}\bigg)h_{K_i}(u)\,d \mu_p(K_{i},u)\\
&=&\frac{p-1}{n-p}\int_{S^{n-1}}\varphi \bigg(\frac{ C\cdot \langle u,v \rangle_+}{R_0}\bigg)h_K(u)\,d \mu_p(K,u)\\
&\geq &\frac{p-1}{n-p}\cdot r_0\cdot\varphi \bigg(\frac{C}{R_0\cdot j_0}\bigg)\,\cdot\int_{\{u\in S^{n-1}:\,\langle u,v \rangle_+\geq\frac{1}{j_0}\}}\langle u,v \rangle_+ \,d \mu_p(K,u).
\end{eqnarray*}
A contradiction $B\geq \infty$ is obtained if let $C\rightarrow \infty$ and hence $\sup_{i\geq 1} R_i<\infty$.
\end{proof}

\subsection{The $p$-capacitary
Orlicz-Petty bodies}
In this subsection, we will investigate the existence, uniqueness and continuity of the $p$-capacitary Orlicz-Petty bodies. Like the polar Orlicz-Minkowski problems in Section \ref{section:3}, we are interested in the following optimization problems for the nonhomoheneous/homogeneous  $L_\varphi$ Orlicz mixed $p$-capacity:
\begin{eqnarray}  &&\sup/\inf \Big\{C_{p,\varphi}(K,L):L\in\mathcal{K}_0 \  \text{and}\ |L^\circ|=\omega_{n}\Big\}, \label{4.5}\\
&&\sup/\inf \Big\{\widehat{C}_{p,\varphi}(K,L):L\in\mathcal{K}_0 \  \text{and}\ |L^\circ|=\omega_{n}\Big\}.\label{4.6}\end{eqnarray}
Our main result is the following theorem which establishes the solvability of $(\ref{4.5})$ and $(\ref{4.6})$ under certain conditions.
\begin{theorem} \label{Theorem existence of petty bodies}
\noindent Let $K \in \mathcal{K}_0$ be a convex body and $\varphi \in \mathcal{I}.$

\noindent(i) There exists a convex body $M\in\mathcal{K}_0$ such that $|M^\circ|=\omega_{n}$ and
\begin{eqnarray}\nonumber
 C_{p,\varphi}(K,M)=\inf \Big\{C_{p,\varphi}(K,L):L\in\mathcal{K}_0 \ \ \mbox{and}\ \ |L^\circ|=\omega_{n}\Big\}.
 \end{eqnarray}
\noindent(ii) There exists a convex body $\widehat{M}\in\mathcal{K}_0$ such that $|\widehat{M}^\circ|=\omega_{n}$ and
$$ \widehat{C}_{p,\varphi}(K,\widehat{M})=\inf \Big\{\widehat{C}_{p,\varphi}(K,L):L\in\mathcal{K}_0 \ \ \mbox{and}\ \ |L^\circ|=\omega_{n}\Big\}.$$
In addition, if $\varphi\in\mathcal{I}$ is convex, then both $M$ and $\widehat{M}$ are unique.
\end{theorem}
\begin{proof}
For convenience, let
\begin{eqnarray}\mathcal{G}_{p,\varphi}^{orlicz}(K)=\inf \Big\{C_{p,\varphi}(K,L):L\in\mathcal{K}_0 \ \ \text{and}\ \ |L^\circ|=\omega_{n}\Big\},\label{definition of G}\\
\widehat{\mathcal{G}}_{p,\varphi}^{orlicz}(K)=\inf \Big\{\widehat{C}_{p,\varphi}(K,L):L\in\mathcal{K}_0 \ \ \text{and}\ \ |L^\circ|=\omega_{n}\Big\}\label{definition of G hat}.\end{eqnarray}

(i) Note that $\mathcal{G}_{p,\varphi}^{orlicz}(K)\leq C_{p,\varphi}(K,B^n_2)<\infty$, due to Definition \ref{nonhomogeneous Orlicz mixed $p$-capacity}. Let $\{M_i\}_{i=1}^\infty \subseteq\mathcal{K}_0$ be the optimal sequence such that
$$C_{p,\varphi}(K,M_i)\rightarrow \mathcal{G}_{p,\varphi}^{orlicz}(K)\ \ \text{and}\ \ |M_i^\circ|=\omega_{n}\ \ \text{for any}\ i\geq1.$$
By Proposition \ref{proposition3.3}, one gets that ${\left\{M_i\right\}}_{i=1}^{\infty}$ is uniformly bounded. By Lemma \ref{blaschke-selection}, one can find a subsequence $\{M_{i_k}\}_{k=1}^\infty$ of $\{M_i\}_{i=1}^\infty$ and $M\in \mathcal{K}_0$ such that
$M_{i_k}\rightarrow M $ as $ k \rightarrow \infty$ and $|M^\circ|=\omega_{n}$.
Thus
$$
\mathcal{G}_{p,\varphi}^{orlicz}(K)=\lim_{i\rightarrow \infty}C_{p,\varphi}(K,M_i)=\lim_{k\rightarrow \infty}C_{p,\varphi}(K,M_{i_k})=C_{p,\varphi}(K,M).
$$
The last identity is due to Proposition \ref{Proposition3.1}. So $M$ is a solution to problem (\ref{4.5}).

(ii) Following along the same lines, one gets a convex body $\widehat{M}\in\mathcal{K}_0$ such that $|\widehat{M}^\circ|=\omega_{n}$ and
$$\widehat{C}_{p,\varphi}(K,\widehat{M})=\widehat{\mathcal{G}}_{p,\varphi}^{orlicz}(K)=\inf \Big\{\widehat{C}_{p,\varphi}(K,L):L\in\mathcal{K}_0 \  \text{and}\ |L^\circ|=\omega_{n}\Big\}.$$

Now we prove the uniqueness of $M.$ Let $M_1$ and $M_2$ be two convex bodies such that $|M_1^\circ|=|M_2^\circ|=\omega_{n}$ and $\mathcal{G}_{p,\varphi}^{orlicz}(K)
=C_{p,\varphi}(K,M_1)=C_{p,\varphi}(K,M_2).$ Let $M_{0}=\displaystyle{\frac{M_1+M_2}{2}}$  
and $\text{vrad}(M^\circ_{0})\leq 1$ with equality if and only if $M_1=M_2$. The fact that $\varphi$ is convex and strictly increasing implies
\begin{eqnarray*} \mathcal{G}_{p,\varphi}^{orlicz}(K)&\leq& C_{p,\varphi}(K,\text{vrad}(M^\circ_{0})\cdot M_{0})\\&=&\frac{p-1}{n-p}\int_{S^{n-1}}\varphi\left( \frac{\text{vrad}(M^\circ_{0})\cdot h_{ M_{0}}(u)}{h_K(u)} \right)h_K(u)\,d \mu_p(K,u)\\&\leq&\frac{p-1}{n-p}\int_{S^{n-1}}\varphi\left( \frac{h_{ M_{0}}(u)}{h_K(u)} \right)h_K(u)\,d \mu_p(K,u)\\
&\leq&\frac{p-1}{n-p}\int_{S^{n-1}}\bigg[\frac{1}{2}\varphi\left( \frac{h_{ M_{1}}(u)}{h_K(u)} \right)h_K(u)+\frac{1}{2}\varphi\left( \frac{h_{ M_{2}}(u)}{h_K(u)} \right)h_K(u)\bigg]\,d \mu_p(K,u)\\&= &\frac{C_{p,\varphi}(K,M_1)+C_{p,\varphi}(K,M_2)}{2}\\
&=&\mathcal{G}_{p,\varphi}^{orlicz}(K).
\end{eqnarray*}
This implies $\text{vrad}(M^\circ_{0})=1$ and hence $M_1=M_2$.

For the uniqueness of $\widehat{M}$, let $\widehat{M}_1$ and $\widehat{M}_2$ be two convex bodies such that $|\widehat{M}_1^\circ|=|\widehat{M}_2^\circ|=\omega_{n}$ and $\widehat{\mathcal{G}}_{p,\varphi}^{orlicz}(K)
=\widehat{C}_{p,\varphi}(K,\widehat{M}_1)=\widehat{C}_{p,\varphi}(K,\widehat{M}_2)$.
Let $\widehat{M}_0= \frac{\widehat{M}_1+\widehat{M}_2}{2}$ and $\text{vrad}(\widehat{M}_0^\circ)\leq 1$ with equality if and only if $\widehat{M}_1=\widehat{M}_2$. By  the convexity of $\varphi$ and the fact that $\widehat{\mathcal{G}}_{p,\varphi}^{orlicz}(K)
=\widehat{C}_{p,\varphi}(K,\widehat{M}_1)=\widehat{C}_{p,\varphi}(K,\widehat{M}_2)$, one has
\begin{eqnarray*}
1&=&\frac{1}{2}\int_{S^{n-1}}\varphi \bigg(\frac{C_p(K)\cdot h_{\widehat{M}_1}(u)}{\widehat{\mathcal{G}}_{p,\varphi}^{orlicz}(K) \cdot h_K(u)}\bigg)\,d \mu^*_p(K,u)+\frac{1}{2}\int_{S^{n-1}}\varphi \bigg(\frac{C_p(K)\cdot h_{\widehat{M}_2}(u)}{\widehat{\mathcal{G}}_{p,\varphi}^{orlicz}(K) \cdot h_K(u)}\bigg)\,d \mu^*_p(K,u)\\&\geq & \int_{S^{n-1}}\varphi \bigg(\frac{C_p(K)\cdot \big (h_{\widehat{M}_1}(u)+h_{\widehat{M}_2}(u)\big)}{2\cdot \widehat{\mathcal{G}}_{p,\varphi}^{orlicz}(K) \cdot h_K(u)}\bigg)\,d \mu^*_p(K,u)\\
&=&\int_{S^{n-1}}\varphi \left(\frac{C_p(K)\cdot h_{\widehat{M}_0}(u)}{\widehat{\mathcal{G}}_{p,\varphi}^{orlicz}(K) \cdot h_K(u)} \right)\,d \mu^*_p(K,u).
\end{eqnarray*}
By Definition \ref{definition-homogeneous} and monotonicity of $\varphi$, one obtains $\widehat{C}_{p,\varphi}(K,\widehat{M}_0)\leq \widehat{\mathcal{G}}_{p,\varphi}^{orlicz}(K)$. Combining this with (\ref{definition of G hat}) and Corollary \ref{corollary3.1}, one has
\begin{eqnarray*}
\widehat{\mathcal{G}}_{p,\varphi}^{orlicz}(K) \leq  \widehat{C}_{p,\varphi}(K,\text{vrad}(\widehat{M}_0^\circ)\cdot \widehat{M}_0) =  \text{vrad}(\widehat{M}_0^\circ)\cdot\widehat{C}_{p,\varphi}(K,\widehat{M}_0) \leq  \widehat{C}_{p,\varphi}(K,\widehat{M}_0) \leq   \widehat{\mathcal{G}}_{p,\varphi}^{orlicz}(K).
\end{eqnarray*}
This yields $\text{vrad}(\widehat{M}_0^\circ)=1$ and hence $\widehat{M}_1=\widehat{M}_2$.
\end{proof}

Theorem \ref{Theorem existence of petty bodies} motivates the following definition of the $p$-capacitary Orlicz-Petty bodies.
\begin{definition}
Let $K \in \mathcal{K}_0$ and $\varphi \in \mathcal{I}.$ Define the set $\mathcal{T}_{p,\varphi}(K)$ to be the collection of all convex bodies $M$ such that $|M^\circ|=\omega_{n}$ and
$$ C_{p,\varphi}(K,M)=\inf \Big\{C_{p,\varphi}(K,L):L\in\mathcal{K}_0 \  \text{and}\ |L^\circ|=\omega_{n}\Big\}.$$
Similarly, let the set $\widehat{\mathcal{T}}_{p,\varphi}(K)$ be the collection of all convex bodies $\widehat{M}$ such that $|\widehat{M}^\circ|=\omega_{n}$ and
$$ \widehat{C}_{p,\varphi}(K,\widehat{M})=\inf \Big\{\widehat{C}_{p,\varphi}(K,L):L\in\mathcal{K}_0 \  \text{and}\ |L^\circ|=\omega_{n}\Big\}.$$
\end{definition}
It follows from Theorem \ref{Theorem existence of petty bodies} that both $\mathcal{T}_{p,\varphi}(K)$ and $\widehat{\mathcal{T}}_{p,\varphi}(K)$ are nonempty if $\varphi\in\mathcal{I}$. When $\varphi\in\mathcal{I}$ is convex, $\mathcal{T}_{p,\varphi}(K)$ and $\widehat{\mathcal{T}}_{p,\varphi}(K)$, respectively,  contain only one element. Consequently, $\mathcal{T}_{p,\varphi}:\mathcal{K}_0\rightarrow\mathcal{K}_0$ and $\widehat{\mathcal{T}}_{p,\varphi}:\mathcal{K}_0\rightarrow\mathcal{K}_0$ define two operators on $\mathcal{K}_0$. The following theorem deals with the continuity of $\mathcal{T}_{p,\varphi}(\cdot),$ $\widehat{\mathcal{T}}_{p,\varphi}(\cdot),$ $\mathcal{G}_{p,\varphi}^{orlicz}(\cdot)$ and $\widehat{\mathcal{G}}_{p,\varphi}^{orlicz}(\cdot).$
\begin{theorem}\label{Theorem 4.2}
Let $\varphi \in \mathcal{I}$ and $\{K_i\}_{i=1}^\infty \subseteq\mathcal{K}_0$ be a sequence converging to $K\in \mathcal{K}_0.$ Then \\
(i) $\mathcal{G}_{p,\varphi}^{orlicz}(K_i)\rightarrow\mathcal{G}_{p,\varphi}^{orlicz}(K)$ and $\widehat{\mathcal{G}}_{p,\varphi}^{orlicz}(K_i)\rightarrow\widehat{\mathcal{G}}_{p,\varphi}^{orlicz}(K)$ as $i \rightarrow \infty$;\\
(ii) if, in addition, $\varphi\in\mathcal{I}$ is convex,  $\mathcal{T}_{p,\varphi}(K_i)\rightarrow\mathcal{T}_{p,\varphi}(K)$ and $\widehat{\mathcal{T}}_{p,\varphi}(K_i)\rightarrow \widehat{\mathcal{T}}_{p,\varphi}(K)$ as $i \rightarrow \infty$.\end{theorem}

\begin{proof}  (i) First of all, we prove $\mathcal{G}_{p,\varphi}^{orlicz}(K_i)\rightarrow\mathcal{G}_{p,\varphi}^{orlicz}(K)$ as $i \rightarrow \infty.$ Let $M\in\mathcal{T}_{p,\varphi}(K)$ and $M_i \in \mathcal{T}_{p,\varphi}(K_i)$ for each $i\geq 1$. By Proposition \ref{Proposition3.1} and (\ref{definition of G}), one has
 \begin{eqnarray}  \nonumber \mathcal{G}_{p,\varphi}^{orlicz}(K)&=&C_{p,\varphi}(K,M)\\  \nonumber&=&\lim_{i\rightarrow \infty}C_{p,\varphi}(K_i,M)\\  \nonumber
 &=&\limsup_{i\rightarrow \infty}C_{p,\varphi}(K_i,M)\\
 &\geq& \limsup_{i\rightarrow \infty}\mathcal{G}_{p,\varphi}^{orlicz}(K_i). \label{continuous-1}
 \end{eqnarray}
This implies that $\{\mathcal{G}_{p,\varphi}^{orlicz}(K_i)\}_{i=1}^\infty$ is bounded. It follows from Proposition \ref{proposition3.3} and $\mathcal{G}_{p,\varphi}^{orlicz}(K_i)=C_{p,\varphi}(K_i,M_i)$ for each $i\geq1$ that $\{M_{i}\}_{i=1}^\infty $ is uniformly bounded. Let $\{K_{i_k}\}_{k=1}^\infty \subseteq\{K_i\}_{i=1}^\infty$ be a subsequence such that $$\lim_{k \rightarrow \infty}\mathcal{G}_{p,\varphi}^{orlicz}(K_{i_k}) = \liminf_{i\rightarrow \infty}\mathcal{G}_{p,\varphi}^{orlicz}(K_i).$$
By the boundedness of $\{M_{i_k}\}_{k=1}^\infty $ and Lemma \ref{blaschke-selection}, there exist a subsequence $\{M_{i_{k_j}}\}_{j=1}^\infty$ of $ \{M_{i_k}\}_{k=1}^\infty$ and $M^\prime\in \mathcal{K}_0$ such that $M_{i_{k_j}}\rightarrow M^\prime$ as $j\rightarrow \infty$ and $|{(M^\prime)}^\circ|=\omega_{n}$. Thus, Proposition \ref{Proposition3.1} yields
 \begin{eqnarray} \label{continuous-2}
 \nonumber \liminf_{i\rightarrow \infty}\mathcal{G}_{p,\varphi}^{orlicz}(K_i)&=&\lim_{j \rightarrow \infty}\mathcal{G}_{p,\varphi}^{orlicz}(K_{i_{k_j}})\\ \nonumber &=&\lim_{j \rightarrow \infty}C_{p,\varphi}(K_{i_{k_j}},M_{i_{k_j}})\\ \nonumber &=&C_{p,\varphi}(K,M^\prime)\\&\geq& \mathcal{G}_{p,\varphi}^{orlicz}(K).
 \end{eqnarray}
From (\ref{continuous-1}) and (\ref{continuous-2}), one concludes that
\begin{equation}\label{continuity of G}
\mathcal{G}_{p,\varphi}^{orlicz}(K)=\lim_{i\rightarrow \infty}\mathcal{G}_{p,\varphi}^{orlicz}(K_i).
\end{equation}
The assertion $\widehat{\mathcal{G}}_{p,\varphi}^{orlicz}(K_i)\rightarrow\widehat{\mathcal{G}}_{p,\varphi}^{orlicz}(K)$ can be proved in a similar manner.

(ii) Next we prove $\mathcal{T}_{p,\varphi}(K_i)\rightarrow\mathcal{T}_{p,\varphi}(K)$ when $\varphi\in\mathcal{I}$ is convex. In this case, by Theorem \ref{Theorem existence of petty bodies}, $\mathcal{T}_{p,\varphi}(K)$ and $\mathcal{T}_{p,\varphi}(K_i)$ contain only one element which will be denoted by $M$ and $M_i$ for each $i\geq 1$.
Let $\{M_{i_k}\}_{k=1}^\infty$ be any subsequence of $\{M_i\}_{i=1}^\infty$.
By the convergence of $K_{i_k} \rightarrow K\in \mathcal{K}_0$ and (\ref{continuity of G}), one has
\be\label{continuous-D-1'}
\mathcal{G}_{p,\varphi}^{orlicz}(K)=\lim_{k\rightarrow\infty}\mathcal{G}_{p,\varphi}^{orlicz}(K_{i_k})
=\lim_{k\rightarrow\infty}C_{p,\varphi}(K_{i_k},M_{i_k}).
\ee
Consequently, $\{C_{p,\varphi}(K_{i_k},M_{i_k})\}_{k=1}^\infty$ is uniformly bounded and then $\{M_{i_k}\}_{k=1}^\infty$ is bounded, due to Proposition \ref{proposition3.3}.
From Lemma \ref{blaschke-selection}, there exist a subsequence $\{M_{i_{k_j}}\}_{j=1}^\infty$ of $\{M_{i_k}\}_{k=1}^\infty$ and a convex body $M^\prime\in \mathcal{K}_0$ such that $M_{i_{k_j}}\rightarrow M^\prime$ and $|{(M^\prime)}^\circ|=\omega_{n}$. By Proposition \ref{Proposition3.1} and (\ref{continuous-D-1'}), one has
$$
\mathcal{G}_{p,\varphi}^{orlicz}(K)
=\lim_{j\rightarrow\infty}\mathcal{G}_{p,\varphi}^{orlicz}(K_{i_{k_j}})=\lim_{j\rightarrow\infty}C_{p,\varphi}(K_{i_{k_j}},M_{i_{k_j}})
=C_{p,\varphi}(K,M^\prime).
$$
Therefore $M=M^\prime$, due to the uniqueness in Theorem \ref{Theorem existence of petty bodies}  if $\varphi\in\mathcal{I}$ is convex. In other words, we have proved that every subsequence of $\{M_i\}_{i=1}^\infty$ has a convergent subsequence with limit of $M.$ Thus $M_i \rightarrow M$ as $i \rightarrow \infty$ with respect to the Hausdorff metric.

Along the same lines, one can prove $\widehat{\mathcal{T}}_{p,\varphi}(K_i)\rightarrow \widehat{\mathcal{T}}_{p,\varphi}(K)$ as $i \rightarrow \infty$ under the condition that $\varphi\in\mathcal{I}$ is convex.
\end{proof}

Again, if $K$ is a polytope, then its $p$-capacitary Orlicz-Petty bodies must be polytopes as well. That is the next proposition.
\begin{proposition} \label{proposition7.1}
 If $K\in \mathcal{K}_0$ is a polytope and $\varphi\in\mathcal{I}$, then $\mathcal{T}_{p,\varphi}(K)$ and $\widehat{\mathcal{T}}_{p,\varphi}(K)$ only contain polytopes with faces parallel to those of $K$.
\end{proposition}
\begin{proof} As $K$ is a polytope, its surface area measure $S(K,\cdot)$ must be concentrated on a finite subset $\{u_1,u_2,\cdots, u_m\}\subseteq S^{n-1}.$ This, together with (\ref{measure-38}), implies that the $p$-capacitary measure $\mu_p(K,\cdot)$ is also concentrated on $\{u_1,u_2,\cdots, u_m\}$ (see e.g. \cite{HYZ2017I}). Let $M\in\widehat{\mathcal{T}}_{p,\varphi}(K)$ and $P$ be a polytope given by (\ref{cir-12-22}). Then, $h_P(u_i)=h_M(u_i)$ $(1\leq i \leq m)$.
Thus, one has
\begin{eqnarray*}
1&=&\int_{S^{n-1}}\varphi \bigg(\frac{C_p(K)\cdot h_{P}(u)}{\widehat{C}_{p,\varphi}(K,P)\cdot h_{K}(u)}\bigg)\,d \mu^*_p(K,u)\\
&=&\frac{p-1}{n-p}\cdot\frac{1}{C_p(K)}\sum^m_{i=1}\varphi\bigg(\frac{C_p(K)\cdot h_{P}(u_i)}{\widehat{C}_{p,\varphi}(K,P)\cdot h_{K}(u_i)}\bigg)\cdot h_K(u_i)\cdot\mu_p(K,\{u_i\})\\
&=&\frac{p-1}{n-p}\cdot\frac{1}{C_p(K)}\sum^m_{i=1}\varphi \bigg(\frac{C_p(K)\cdot h_{M}(u_i)}{\widehat{C}_{p,\varphi}(K,P)\cdot h_{K}(u_i)}\bigg)\cdot h_K(u_i)\cdot\mu_p(K,\{u_i\})\\
&=&\int_{S^{n-1}}\varphi \bigg(\frac{C_p(K)\cdot h_{M}(u)}{\widehat{C}_{p,\varphi}(K,P)\cdot h_{K}(u)}\bigg)\,d \mu^*_p(K,u).
\end{eqnarray*}
This yields $\widehat{C}_{p,\varphi}(K,P)=\widehat{C}_{p,\varphi}(K,M)$. On the other hand, by (\ref{definition of G hat}) and Corollary \ref{corollary3.1}, one gets
$$
\widehat{C}_{p,\varphi}(K,P)=\widehat{C}_{p,\varphi}(K,M)=\widehat{\mathcal{G}}_{p,\varphi}^{orlicz}(K)\leq \widehat{C}_{p,\varphi}(K,\text{vrad}(P^\circ)P)=\text{vrad}(P^\circ)\cdot\widehat{C}_{p,\varphi}(K,P).
$$
This implies $\text{vrad}(P^\circ)\geq1.$  Due to $P^\circ \subseteq M^\circ$, one has $\text{vrad}(P^\circ)\leq \text{vrad}(M^\circ)=1$. Hence $|P^\circ|=|M^\circ|$ and then $M=P$.

Employing the same argument, one can prove that each $M\in \mathcal{T}_{p,\varphi}(K)$ is a polytope with faces parallel to those of $K$.
\end{proof}

The following proposition can be proved by the same techniques as the proofs of Proposition \ref{denial of other possibilities}. From this proposition, one sees that problems $(\ref{4.5})$ and $(\ref{4.6})$ may not be solvable in general except the case stated in Theorem \ref{Theorem existence of petty bodies}.
\begin{proposition}\label{proposition 4.4}Let $K\in \mathcal{K}_0$ be a polytope and $S(K,\cdot)$ be its surface area measure on $S^{n-1}$ which is concentrated on a finite subset $\{u_1,u_2,\cdots, u_m\}\subseteq S^{n-1}.$\\
(i) If $\varphi\in\mathcal{I},$ then \begin{eqnarray*} \ \ \ \ \sup\Big\{C_{p,\varphi}(K,L):L\in\mathcal{K}_0 \  \text{and}\ |L^\circ|=\omega_{n}\Big\}= \sup \Big\{\widehat{C}_{p,\varphi}(K,L):L\in\mathcal{K}_0 \  \text{and}\ |L^\circ|=\omega_{n}\Big\}=\infty.\end{eqnarray*} 
  (ii) If $\varphi\in\mathcal{D}$, then 
\begin{eqnarray*} &&\inf \Big\{\widehat{C}_{p,\varphi}(K,L):L\in\mathcal{K}_0 \  \text{and}\ |L^\circ|=\omega_{n}\Big\}=0,\\
&&\sup\Big\{C_{p,\varphi}(K,L):L\in\mathcal{K}_0 \  \text{and}\ |L^\circ|=\omega_{n}\Big\}=\infty.\end{eqnarray*} 
(iii) If $\varphi\in\mathcal{D}$ and the jth ($1\leq j\leq n$) coordinates of $u_1,u_2,\cdots, u_m$ are nonzero, then 
\begin{eqnarray*} && \inf \Big\{C_{p,\varphi}(K,L):L\in\mathcal{K}_0 \  \text{and}\ |L^\circ|=\omega_{n}\Big\}=0,\\  &&\sup\Big\{\widehat{C}_{p,\varphi}(K,L):L\in\mathcal{K}_0 \  \text{and}\ |L^\circ|=\omega_{n}\Big\}=\infty.\end{eqnarray*} \end{proposition}
 
 In fact, we can replace $|L^\circ|$ in problems $(\ref{4.5})$ and $(\ref{4.6})$ by $C_p(L^\circ)$ and consider the following optimization problems:
\begin{eqnarray}&&\sup/\inf \Big\{\widehat{C}_{p,\varphi}(K,L):L\in\mathcal{K}_0 \  \text{and}\ C_p(L^\circ)=C_p(B_2^n)\Big\}, \label{4.12}\\
&&\sup/\inf \Big\{{C}_{p,\varphi}(K,L):L\in\mathcal{K}_0 \  \text{and}\ C_p(L^\circ)=C_p(B_2^n)\Big\}\nonumber.\end{eqnarray}
 The following result can be obtained.
\begin{theorem}
\noindent Let $K \in \mathcal{K}_0$ and $\varphi \in \mathcal{I}.$\\
(i) There exists a convex body $\widehat{M}\in\mathcal{K}_0$ such that  $C_p(\widehat{M}^\circ)=C_p(B_2^n)$ and
$$\widehat{\mathcal{H}}_{p,\varphi}^{orlicz}(K)=\inf \Big\{\widehat{C}_{p,\varphi}(K,L):L\in\mathcal{K}_0 \  \text{and}\ C_p(L^\circ)=C_p(B_2^n)\Big\}=\widehat{C}_{p,\varphi}(K,\widehat{M}).$$
 Moreover, if $\{K_i\}^\infty_{i=1}\subseteq\mathcal{K}_0$ satisfies $K_i\rightarrow K,$ then $\widehat{\mathcal{H}}_{p,\varphi}^{orlicz}(K_i)\rightarrow \widehat{\mathcal{H}}_{p,\varphi}^{orlicz}(K).$
 \vskip 2mm \noindent
(ii) There exists a convex body $M\in\mathcal{K}_0$ such that  $C_p(M^\circ)=C_p(B_2^n)$ and
$$
\mathcal{H}_{p,\varphi}^{orlicz}(K)=\inf \Big\{C_{p,\varphi}(K,L): L\in\mathcal{K}_0 \  \text{and}\ C_p(L^\circ)=C_p(B_2^n)\Big\}=C_{p,\varphi}(K,M).$$
Moreover, if $\{K_i\}^\infty_{i=1}\subseteq\mathcal{K}_0$ satisfies $K_i\rightarrow K,$ then $\mathcal{H}_{p,\varphi}^{orlicz}(K_i)\rightarrow \mathcal{H}_{p,\varphi}^{orlicz}(K).$ 
\end{theorem}

\begin{proof} (i) Let $\{M_i\}_{i=1}^\infty \subseteq\mathcal{K}_0$ be a sequence of convex bodies such that
$$
\widehat{C}_{p,\varphi}(K, M_i)\rightarrow \widehat{\mathcal{H}}_{p,\varphi}^{orlicz}(K)\ \ \text{and}\ \  C_p({M_i}^\circ)=C_p(B_2^n)\,\,\,\text{for any}\,\, i\geq 1.
$$
As $\widehat{\mathcal{H}}_{p,\varphi}^{orlicz}(K)\leq\widehat{C}_{p,\varphi}(K,B^n_2)<\infty,$ then $\{\widehat{C}_{p,\varphi}(K,M_i)\}_{i=1}^\infty$ is bounded. Proposition \ref{proposition3.3} implies that $\{M_i\}_{i=1}^\infty$ is uniformly bounded.  Due to $C_p({M}^\circ_{i})=C_p(B_2^n)$ for any $i\geq1$, it follows from  (\ref{measure-36}) that $\{|M^\circ_{i}|\}^\infty_{i=1}$ is bounded. By Lemma \ref{blaschke-selection}, there exist a subsequence $\{M_{i_k}\}_{k=1}^\infty$ of $\{M_i\}_{i=1}^\infty$ and  $\widehat{M}\in \mathcal{K}_0$ such that
$M_{i_k}\rightarrow \widehat{M} $ as $ k \rightarrow \infty.$ By Proposition \ref{Proposition3.1}, one has $C_p(\widehat{M}^\circ)=\lim_{k\rightarrow \infty}C_p({M}^\circ_{i_k})=C_p(B_2^n)$ and
$$\widehat{\mathcal{H}}_{p,\varphi}^{orlicz}(K)=\lim_{i\rightarrow \infty}\widehat{C}_{p,\varphi}(K,M_i)=\lim_{k\rightarrow \infty}\widehat{C}_{p,\varphi}(K, M_{i_k})=\widehat{C}_{p,\varphi}(K,\widehat{M}).$$
Thus, $\widehat{M}$ is a solution to problem (\ref{4.12}).

Let $\{\widehat{M}_i\}^\infty_{i=1}$ be a sequence of convex bodies such that
$$
\widehat{\mathcal{H}}_{p,\varphi}^{orlicz}(K_i)=\widehat{C}_{p,\varphi}(K_i,\widehat{M}_i)\ \ \text{and}\ \  C_p({\widehat{M}_i}^\circ)=C_p(B_2^n)\,\,\,\text{for any}\,\, i\geq 1.
$$
By Proposition \ref{Proposition3.1}, one has
 \begin{eqnarray}  \nonumber\widehat{\mathcal{H}}_{p,\varphi}^{orlicz}(K)&=&\widehat{C}_{p,\varphi}(K,\widehat{M})\\ \nonumber &=&\lim_{i\rightarrow \infty}\widehat{C}_{p,\varphi}(K_i,\widehat{M})\\  \nonumber
 &=&\limsup_{i\rightarrow \infty}\widehat{C}_{p,\varphi}(K_i,\widehat{M})\\
 &\geq& \limsup_{i\rightarrow \infty}\widehat{\mathcal{H}}_{p,\varphi}^{orlicz}(K_i). \label{continuous-1'}
 \end{eqnarray}
This, together with Proposition \ref{proposition3.3}, implies that $\{\widehat{M}_i\}^\infty_{i=1}$ is bounded. Let $\{K_{i_k}\}_{k=1}^\infty \subseteq\{K_i\}_{i=1}^\infty$ be a subsequence such that
$$\lim_{k \rightarrow \infty}\widehat{\mathcal{H}}_{p,\varphi}^{orlicz}(K_{i_k}) = \liminf_{i\rightarrow \infty}\widehat{\mathcal{H}}_{p,\varphi}^{orlicz}(K_i).$$
By the boundedness of $\{\widehat{M}_{i_k}\}^\infty_{k=1}$ and Lemma \ref{blaschke-selection}, together with (\ref{measure-36}) and $C_p(\widehat{M}^\circ_{i_k})=C_p(B_2^n)$ for any $k\geq1$, one can find a subsequence $\{\widehat{M}_{i_{k_j}}\}^\infty_{j=1}$ of $\{\widehat{M}_{i_k}\}^\infty_{k=1}$ and $\widehat{M}_0\in\mathcal{K}_0$ such that $\widehat{M}_{i_{k_j}}\rightarrow \widehat{M}_0$ as $j\rightarrow \infty $ and $C_p\big({\widehat{M}_0}^\circ\big)=C_p(B_2^n)$. Thus, by Proposition \ref{Proposition3.1} again, one gets
$$\liminf_{i\rightarrow \infty}\widehat{\mathcal{H}}_{p,\varphi}^{orlicz}(K_i)=\lim_{j \rightarrow \infty}\widehat{\mathcal{H}}_{p,\varphi}^{orlicz}(K_{i_{k_j}})= \lim_{j \rightarrow \infty}\widehat{C}_{p,\varphi}(K_{i_{k_j}},\widehat{M}_{i_{k_j}})=\widehat{C}_{p,\varphi}(K,\widehat{M}_0)\geq \widehat{\mathcal{H}}_{p,\varphi}^{orlicz}(K).$$
Combing this with (\ref{continuous-1'}), one has $\widehat{\mathcal{H}}_{p,\varphi}^{orlicz}(K_i)\rightarrow \widehat{\mathcal{H}}_{p,\varphi}^{orlicz}(K).$ 

The assertation (ii) follows along the same lines.
\end{proof}

The $p$-capacitary measure $\mu_p(K,\cdot)$ in problems (\ref{4.5}) and (\ref{4.6}) could be replaced by the measure $S_\mathcal{V}(K,\cdot)$ given by Definition \ref{definition 2.1}.  In fact,  Hong, Ye and Zhu in \cite{HYZ2017} proposed the following $L_\varphi$ Orlicz mixed $\mathcal{V}$-measure of $K, L\in \mathcal{K}_0$:
$$\mathcal{V}_\varphi(K,L)=\int_{S^{n-1}}\varphi \bigg(\frac{h_L(u)}{h_K(u)}\bigg) h_K(u)\,d S_\mathcal{V}(K,u),$$
where $\varphi\in\mathcal{I}\cup\mathcal{D}$.
For $\varphi\in\mathcal{I}\cup\mathcal{D},$ one can define $\widehat{\mathcal{V}}_\varphi(K,L)$ by
$$\int_{S^{n-1}}\varphi\bigg(\frac{\mathcal{V}(K)\cdot h_L(u)}{\widehat{\mathcal{V}}_\varphi(K,L)\cdot h_K(u)}\bigg)  h_K(u)\,d S_\mathcal{V} (K,u)=\int_{S^{n-1}} h_K(u)\,d S_\mathcal{V} (K,u).$$ The following theorem can be proved similar to the proof of Theorem \ref{Theorem existence of petty bodies}.
\begin{theorem}
Let $K \in \mathcal{K}_0$ and $\varphi \in \mathcal{I}.$ There exists a convex body $M\in\mathcal{K}_0$ such that $|M^\circ|=\omega_{n}$ and
$$\mathcal{V}_\varphi(K,M)=\inf \Big\{\mathcal{V}_\varphi(K,L):L \in \mathcal{K}_{0}\ \text{and}\ |L^\circ|=\omega_{n}\Big\}.$$
Moreover, if $\varphi\in\mathcal{I}$ is convex, then $M$ is unique.\label{Theorem9.1}

Similarly, there exists a convex body $\widehat{M}\in\mathcal{K}_0$ such that $|\widehat{M}^\circ|=\omega_{n}$ and
$$\widehat{\mathcal{V}}_\varphi(K,\widehat{M})=\inf \Big\{\widehat{\mathcal{V}}_\varphi(K,L):L \in \mathcal{K}_{0}\ \text{and}\ |L^\circ|=\omega_{n}\Big\}.$$
If $\varphi\in\mathcal{I}$ is convex, then $\widehat{M}$ is unique.
\end{theorem}
Besides, results in Propositions \ref{proposition7.1} and \ref{proposition 4.4} as well as in Theorem \ref{Theorem 4.2} can be obtained for the case of variational functionals. We leave the details for readers.

\vskip 2mm \noindent {\bf Acknowledgments.} 
DY is supported by a NSERC
grant. BZ is
supported by AARMS, NSERC, NSFC (No.\ 11501185) and the Doctor
Starting Foundation of Hubei University for Nationalities (No.\
MY2014B001).  The authors are greatly indebted to the referee for many valuable comments which improve largely the quality of the paper.

\vskip 2mm \noindent Xiaokang Luo, \ \ \ {\small \tt xl3410@mun.ca}\\
{ \em Department of Mathematics and Statistics,
   Memorial University of Newfoundland,
   St.\ John's, Newfoundland, Canada A1C 5S7 }

\vskip 2mm \noindent Deping Ye, \ \ \ {\small \tt deping.ye@mun.ca}\\
{ \em Department of Mathematics and Statistics,
   Memorial University of Newfoundland,
   St.\ John's, Newfoundland, Canada A1C 5S7 }

\vskip 2mm \noindent Baocheng Zhu, \ \ \ {\small \tt zhubaocheng814@163.com}\\
{ \em 1. Department of Mathematics,
 Hubei University for Nationalities,
 Enshi, Hubei, China 445000}\\
{  \em 2.\  Department of Mathematics and Statistics,   Memorial University of Newfoundland,
   St.\ John's, Newfoundland, Canada A1C 5S7 }

\end{document}